\documentclass[runningheads,numbook]{svjour3}
\usepackage{amsmath,amssymb}
\usepackage{ntheorem}
\usepackage{cite}
\usepackage[left=4.5cm, right=4.5cm, top=4cm, bottom=4cm]{geometry}
\usepackage{color}
\usepackage{hyperref}
\hypersetup{
	colorlinks=true,
	linkcolor=blue,     citecolor=red,     urlcolor=blue  }
\usepackage[colorinlistoftodos]{todonotes}
\usepackage{mathptmx}   
\usepackage[weather,alpine,misc,geometry]{ifsym}
\usepackage{amsfonts}
\usepackage{graphicx}%
\usepackage[capitalise,nameinlink]{cleveref}
\crefformat{equation}{(#2#1#3)}
\crefmultiformat{equation}{(#2#1#3)}%
{ and~(#2#1#3)}{, (#2#1#3)}{ and~(#2#1#3)}
\crefrangeformat{equation}{(#3#1#4)--(#5#2#6)}
\crefrangemultiformat{equation}{(#3#1#4)--(#5#2#6)}%
{ and~(#3#1#4)--(#5#2#6)}{, (#3#1#4)--(#5#2#6)}{ and~(#3#1#4)--(#5#2#6)}

\newcommand{\al}{\alpha}

\newcommand{\de}{\delta}
\newcommand{\eps}{\varepsilon}
\newcommand{\bx}{\bar x}
\newcommand{\by}{\bar y}
\newcommand{\bz}{\bar z}

\newcommand{\iv}{^{-1} }

\newcommand {\R} {\mathbb R}
\newcommand {\N} {\mathbb N}

\newcommand {\B} {\mathbb B}
\newcommand {\Sp} {\mathbb S}
\newcommand {\gph} {{\rm gph}\,}
\newcommand {\dom} {{\rm dom}\,}

\newcommand {\cl} {{\rm cl}\,}

\newcommand {\Int} {{\rm int}\,}

\newcommand{\toto}{\rightrightarrows}
\renewcommand{\iff}{$ \Leftrightarrow\ $}
\newcommand{\folgt}{$ \Rightarrow\ $}

\newcommand{\vertiii}[1]{\left\vert\kern-0.25ex\left\vert\kern-0.25ex\left\vert #1\right\vert\kern-0.25ex\right\vert\kern-0.25ex\right\vert}
\newcommand{\vertiiiBig}[1]{\Big\vert\kern-0.25ex\Big\vert\kern-0.25ex\Big\vert #1\Big\vert\kern-0.25ex\Big\vert\kern-0.25ex\Big\vert}

\def\es{\emptyset}

\def\SVM{set-valued mapping}

\def\Fr{Fr\'echet}

\newcommand{\blue}[1]{\textcolor{blue}{#1}}

\newcommand{\red}[1]{\textcolor{red}{#1}}

\newcommand{\ang}[1]{\left\langle #1 \right\rangle}
\newcommand{\qdtx}[1]{\quad\mbox{#1}\quad}
\newcommand{\AND}{\quad\mbox{and}\quad}
\newcounter{mycount}

\newcommand{\AK}[1]{\todo[inline]{AK {#1}}}


\smartqed
\renewcommand\theenumi{\roman{enumi}}

\renewcommand{\red}[1]{#1}

\newcommand{\Thao}[1]{\todo[inline,color=magenta!40]{Thao {#1}}}
\newcommand{\NDC}[1]{\todo[inline,color=green!40]{NDC {#1}}}

\begin{document}
\title{
Extremality of Families of Sets and Set-Valued Optimization
}
\author{
Nguyen Duy Cuong
\and
Alexander Y. Kruger \and Nguyen Hieu Thao}
\titlerunning{Extremality of Families of Sets and Set-Valued Optimization}
	
\institute{
Nguyen Duy Cuong \at
Department of Mathematics, College of Natural Sciences, Can Tho University, Can Tho \red{City}, Vietnam\\
\email{ndcuong@ctu.edu.vn}, ORCID: 0000-0003-2579-3601
\and
A. Y. Kruger (corresponding author) \at
Optimization Research Group,
Faculty of Mathematics and Statistics,
Ton Duc Thang University,
Ho Chi Minh City, Vietnam \\
\email{alexanderkruger@tdtu.edu.vn}, ORCID: 0000-0002-7861-7380
\and
Nguyen Hieu Thao \at
School of Science, Engineering and Technology, RMIT University Vietnam, Ho Chi Minh City, Vietnam\\
\email{thao.nguyenhieu@rmit.edu.vn}, ORCID: 0000-0002-1455-9169}
\dedication{\red{Dedicated to R. Tyrrell Rockafellar, a trailblazer in optimization and analysis, on his 90th birthday}}
\maketitle

\begin{abstract}
\red{The paper continues our recent work \cite{CuoKruTha24} where another extension of the \emph{extremal principle} has been established.}
We demonstrate its applicability to set-valued optimization problems with general preferences, weakening the assumptions of the known results and streamlining their proofs.
\end{abstract}



\setcounter{tocdepth}{5}

\section{Introduction}

\red{The paper continues our recent work \cite{CuoKruTha24} where a new extremality model involving collections of arbitrary families of sets has been studied and another extension of the \emph{extremal principle} has been established.}

\red{We consider applications of the latter result to set-valued optimization problems of the type
\begin{gather}
\label{P}
\tag{$P$}
\text{minimize }\;F(x)\quad \text{subject to }\; x\in\Omega,
\end{gather}
where $F:X\toto Y$ is a \SVM\ between normed vector spaces, $\Omega$ is a subset of $X$, and the space $Y$ is equipped with a general preference relation determined by an abstract \emph{level-set mapping} $L:Y\toto Y$, as well as more general problems with set-valued constraints.
}

\red{The conventional special case of importance is when $Y$ is equipped with a partial order determined by a nontrivial pointed convex cone, and optimality is understood in the sense of \emph{Pareto}.
\begin{definition}
[Pareto optimality]
\label{Pareto}
Let $X$ and $Y$ be normed spaces, $F:X\toto Y$, $\Omega\subset X$, $\bx\in\Omega$ and $\by\in F(\bx)$.
The point $(\bx,\by)$ is a (local) Pareto solution to \eqref{P} with respect to a nontrivial pointed convex cone $K\subset Y$ if there is a $\de\in(0,+\infty]$ such that
$F(\Omega\cap B_\de(\bx))\cap(\by-K)=\{\by\}$.
\end{definition}
\if{
\begin{remark}
The cone $K$ in Definition~\ref{Pareto} is usually
assumed closed, in which case the last inclusion is equivalent to the equality $F(\Omega\cap B_\de(\bx))\cap(\by-K)=\{\by\}$.
\end{remark}
}\fi
}

The conventional extremal principle \cite{KruMor80,MorSha96,Mor06.1} covers a wide range of problems in optimization and variational analysis as demonstrated, e.g., in the books \cite{BorZhu05,Mor06.1,Mor06.2}.
\red{The advantages of employing the extremal principle as the main tool when deducing necessary optimality conditions in vector and set-valued optimization problems compared to the scalarization and other traditional techniques were emphasized in \cite[Sect.~5.3 and 5.5.18]{Mor06.2}.
At the same time, there exist multiobjective problems with more general
preference relations, which ``may go
far beyond generalized Pareto/weak Pareto concepts of optimality'' \cite[p.~70]{Mor06.2} that cannot be covered by traditional techniques or within the framework of the conventional extremal principle using linear translations.}
The first example of this kind was identified in Zhu \cite{Zhu00}.
Fortunately, such problems can be handled with the help of a more flexible extended version of the extremal principle using nonlinear perturbations (deformations) of the sets defined by set-valued mappings.
Such an extension was developed in Mordukhovich et al. \cite{MorTreZhu03} (see also \cite{BorZhu05,Mor06.2}) and applied
to various multiobjective problems \cite{ZheNg06,LiNgZhe07,Bao14.2}.
\red{Below is our interpretation of the corresponding definitions from \cite{MorTreZhu03,Mor06.2} complying with the notation and terminology adopted in the current paper.
\begin{definition}
[Extremality: set-valued perturbations]
\label{D1.4}
Let $\Omega_1,\ldots,\Omega_n$ be subsets of a normed space $X$, $\bx\in\bigcap_{i=1}^n\Omega_i$, and,
for each $i=1,\ldots,n$,
$S_i:M_i\toto X$ be a \SVM\ from a metric space
$(M_i,d_i)$ to $X$ and $S_i(\bar s_i)=\Omega_i$ for some $\bar s_i\in M_i$.
The collection $\{\Omega_1,\ldots,\Omega_n\}$
is extremal at $\bx$ with respect to $\{S_1,\ldots,S_n\}$ if there exists a $\rho\in(0,+\infty]$ such that, for any $\varepsilon>0$, there exist $s_i\in M_i$ $(i=1,\ldots,n)$ such that
$\max_{1\le i\le n}d_i(s_i,\bar s_i)<\varepsilon$,
$\max_{1\le i\le n}d(\bx,S_i(s_i))<\varepsilon$
and
$\bigcap_{i=1}^nS_i(s_i)\cap B_\rho(\bx){=\emptyset}$.
\end{definition}
}

The model in
\red{Definition~\ref{D1.4}}
exploits non-inter\-section of perturbations of given sets $\Omega_1,\ldots,\Omega_n$.
The perturbations are chosen from the respective families of sets $\Xi_i:=\{S_i(s)\mid s\in M_i\}$
$(i=1,\ldots,n)$
determined by given set-valued mappings $S_i:M_i\toto X$
$(i=1,\ldots,n)$ on metric spaces.
In the particular case of linear translations, i.e., when, for all $i=1,\ldots,n$, $(M_i,d_i)=(X,d)$
and $S_i(a)=\Omega_i-a$ $(a\in X)$, the model reduces to the conventional extremal principle.
It was shown by examples in \cite{MorTreZhu03,Mor06.2} that the framework of set-valued perturbations is richer than that of linear translations.
With minor modifications in the proof, the conventional extremal principle was extended to the set-valued setting producing a more advanced model.

\red{Definition~\ref{D1.4} talks about extremality of a collection of sets, but in fact it is about certain properties of a collection of \SVM s
$S_i:M_i\toto X$ $(i=1,\ldots,n)$, loosely connected with the given sets.}
This model has been refined in \cite{CuoKruTha24}, making it more flexible and, at the same time, simpler.
\red{Instead of the \SVM s $S_1,\ldots,S_n$,}
the refined model studies extremality and stationarity of nonempty families
\red{$\Xi_1,\ldots,\Xi_n$
	of arbitrary sets}
and is applicable to a wider range of variational problems.
The next definition and theorem are simplified versions of \cite[Definition~3.2 and Theorem~3.4]{CuoKruTha24}, respectively.

\begin{definition}
[Extremality and stationarity: families of sets]
\label{D1.5}
Let $\Xi_1,\ldots,\Xi_n$ be families of subsets of a normed space $X$, and $\bx\in X$.
The collection $\{\Xi_1,\ldots,\Xi_ n\}$ is
\begin{enumerate}
\item
\label{D1.5.1}
extremal at $\bx$ if there is a $\rho\in(0,+\infty]$ such that, for any $\varepsilon>0$, there exist $A_i\in\Xi_i$ $(i=1,\ldots,n)$ such that $\max_{1\le i\le n}d(\bx,A_i)< \varepsilon$ and
$\bigcap_{i=1}^nA_i\cap B_\rho(\bx)=\emptyset$;
\item
\label{D1.5.3}
stationary at $\bx$ if, for any $\varepsilon>0$, there exist a $\rho\in(0,\varepsilon)$ and  $A_i\in\Xi_i$ $(i=1,\ldots,n)$ such that
$\max_{1\le i\le n}d(\bx, A_i)<\varepsilon\rho$ and
$\bigcap_{i=1}^nA_i\cap B_\rho(\bx)=\emptyset$;
\item
\label{D1.5.2}
approximately stationary at $\bx$ if, for any $\varepsilon>0$, there exist a $\rho\in(0,\varepsilon)$, $A_i\in\Xi_i$ and $x_i\in B_{\varepsilon}(\bx)$ $(i=1,\ldots,n)$ such that
$\max_{1\le i\le n}d(x_i, A_i)<\varepsilon\rho$ and
$\bigcap_{i=1}^n(A_i-x_i)\cap(\rho\B)=\emptyset$.
\end{enumerate}
\end{definition}

\begin{theorem}
\label{T1.4}
Let
$\Xi_1,\ldots,\Xi_n$ be families of closed subsets of a Banach space $X$, and $\bx\in X$.
If $\{\Xi_1,\ldots,\Xi_n\}$ is approximately stationary at $\bx$,
then, for any $\varepsilon>0$, there exist $A_i\in\Xi_i$,
$x_i\in A_i\cap B_\varepsilon(\bx)$, and $x_i^*\in N^C_{A_i}(x_i)$ $(i=1,\ldots,n)$ such that
\begin{gather*}
\Big\|\sum_{i=1}^nx_i^*\Big\|<\eps
\qdtx{and}
\sum_{i=1}^n\|x_i^*\|=1.
\end{gather*}

If $X$ is Asplund, then $N^C$ in the above assertion can be replaced by $N^F$.
\end{theorem}

The symbols $N^C$ and $N^F$ in the above theorem denote, respectively, the Clarke and \Fr\ normal cones.
Recall that a Banach space is \emph{Asplund} if every continuous convex function on an open convex set is Fr\'echet differentiable on a dense subset \cite{Phe93}, or equivalently, if the dual of each its separable subspace is separable.
We refer the reader to \cite{Phe93,Mor06.1} for discussions about and characterizations of Asplund spaces.
All reflexive, particularly, all finite dimensional Banach spaces are Asplund.
Most assertions involving \Fr\ normals, subdifferentials and coderivatives are only valid in
Asplund spaces; see \cite{MorSha96}.

\if{
\AK{6/08/24.
I have removed the ``additional'' conditions from the statement as we do not use them in the set-valued part.
As a result, $\Omega_1,\ldots,\Omega_n$ are not involved in the conclusions.
This does not seem right.
To be checked.}
\NDC{9/8/24.
I think when $\Xi_i:=\{\Omega_i\}$, then Theorem \ref{T1.4} recaptures the conventional extremal principle.
However, Theorem~\ref{T1.4} does not require $\bx\in\bigcap_{i=1}^n\Omega_i$.
}
\NDC{9/8/24.
It still not right. When $\Xi_i:=\{\Omega_i\}$, the approximate stationarity in Definition~\ref{D1.5} does not reduces to the conventional one.}
\NDC{9/8/24.
I think in a particular case when $\Xi_i:=\{\Omega_i-a\mid a\in X\}$, Theorem~\ref{T1.4} reduces to the conventional extremal principle.
Concerning the missing of $\Omega_i$ in the conclusion part of Theorem~\ref{T1.4}, I think the similar situation happens for the case of set-valued extremal principle, i.e., the dual vectors are chosen from $N^F_{S_i(s_i)}(x_i)$ and the sets $ S_i(\bar s_i)$ disappear; the only relevance between $S_i(s_i)$ and $S_i(\bar s_i)$ is the requirement $d(s_i,\bar s_i)<\varepsilon$ which seems to be minor of importance since the definition of extremality/stationarity does not require any continuity assumptions of the set-valued mappings $S_i$.
}
\AK{10/08/24.
It looks like we can take $\Omega_1=\ldots=\Omega_n:=X$ in Theorem~\ref{T1.4}.
This is the weakest version of Definition~\ref{D1.5}\,\ref{D1.5.2}.
Then $\Omega_1,\ldots,\Omega_n$ can be dropped completely.}
\NDC{10/8/24.
I think we can add item (iv) in Definition~\ref{D1.5} about `approximately stationary at $\bx$' where it requires $x_i\in B_\varepsilon(\bx)$ instead of $x_i\in\Omega_i\cap B_\varepsilon(\bx)$.
I think a comment about the additional condition of the dual vectors should be added.}
}\fi

\begin{remark}
\begin{enumerate}
\item
Part \eqref{D1.5.3} of
Definition~\ref{D1.5} is the explicit form of
\cite[Definition~3.2\,(iii)]{CuoKruTha24}, while part
\eqref{D1.5.2} is a particular case of \cite[Definition~3.2\,(ii)]{CuoKruTha24}
with $\Omega_1=\ldots=\Omega_n:=X$,
thus, representing the weakest version of the property in \cite[Definition~3.2\,(ii)]{CuoKruTha24}.
\item
It is easy to see that \eqref{D1.5.1} \folgt \eqref{D1.5.3} \folgt \eqref{D1.5.2} in Definition~\ref{D1.5}.
Hence, the necessary conditions in Theorem~\ref{T1.4} are also valid for the stationarity and extremality.
\item
Theorem~\ref{T1.4} shows that approximate stationarity of a given collection of families of closed sets implies its fuzzy (up to $\eps$) separation.
Note that, unlike the general model discussed in \cite{BuiKru19}, not only the points $x_i$ and $x_i^*$ $(i=1,\ldots,n)$ usually involved in fuzzy separation statements depend on $\eps$, but also the
sets $A_1,\ldots,A_n$.
\item
For each $i=1,\ldots,n$, the sets $A_i$ making the family $\Xi_i$ in Definition~\ref{D1.5} can be considered as perturbations of some given set $\Omega_i$.
With this interpretation in mind, Definition~\ref{D1.5}\,\eqref{D1.5.1} covers
\red{Definition~\ref{D1.4}}.
Note that the ``perturbation'' sets in
\red{Definition~\ref{D1.4}}
are rather loosely connected with the given sets.
\end{enumerate}
\end{remark}

\begin{example}
Let $\Xi_1$ consist of a single one-point set $\{0\}\subset\R$,  and $\Xi_2$ be a family of singletons $\{1/n\}$ for $n\in\N$.
It is easy to see that
$\{\Xi_1,\Xi_2\}$ is extremal at $0$ in the sense of Definition~\ref{D1.5}\,\eqref{D1.5.1} (even with $\rho=+\infty$).
The subsets of $\Xi_1,\Xi_2$ may be considered as ``perturbations'' of the sets $\Omega_1=\Omega_2:=\R$ in the sense of
\red{Definition~\ref{D1.4}}.
The pair $\{\Omega_1,\Omega_2\}$ is clearly not extremal at $0$ in the conventional sense of \cite{KruMor80,Mor06.1}.
\end{example}

The more general (and simpler) model in Definition~\ref{D1.5} and Theorem~\ref{T1.4} is capable of treating a wider range of
applications.
In this paper, we demonstrate the applicability of Theorem~\ref{T1.4} to set-valued optimization problems
\red{of the type \eqref{P} and more general ones}.
This allows us to expand the range of set-valued optimization models studied in earlier publications, weaken their assumptions and streamline the proofs.

We study extremality/stationarity properties of the triple $\{F,\Omega,\Xi\}$, where $F:X\toto Y$ is a \SVM\ between normed spaces, $\Omega$ is a subset of $X$, and $\Xi$ is a nonempty family of subsets of $Y$.
The latter family may, in particular be determined
by an abstract \emph{level-set mapping} defining a \emph{preference} relation on $Y$.

Members of $\Xi$ do not have to be simply translations (deformations) of a fixed set (ordering cone).
Extremality/stationarity properties of the triple $\{F,\Omega,\Xi\}$ reduce to the corresponding properties of the two special families of subsets of $X\times Y$:
\red{
\begin{gather}
\label{Xi}
\Xi_1:=\{\gph F\}\quad\text{and}\quad \Xi_2:=\{\Omega\times A\mid A\in\Xi\}.
\end{gather}
The first family consists of the single set $\gph F$, and the first component of each member of the second family is always the given set $\Omega$; only the second component varies.}

The properties are illustrated by examples.
Application of Theorem~\ref{T1.4} yields necessary conditions for approximate stationarity and, hence, also stationarity and extremality.
Natural \emph{qualification conditions} in terms of Clarke or \Fr\ coderivatives and normal cones are provided, which allow one to write down the necessary conditions in the form of an abstract \emph{multiplier rule}.
The statements cover the corresponding results in
\cite{MorTreZhu03,ZheNg05.2,ZheNg06}.

Requirements on {preference} relations defined by level-set mappings, making them meaningful in optimization and applications, are discussed.
A certain subset of properties, which are satisfied by most conventional and many other preference relations, is established.
The properties are shown to be in general weaker than those used in \cite{Mor06.2,Bao14.2,KhaTamZal15}, but still sufficient for the corresponding set-valued optimization problems to fall within the theory developed in the current paper.
Several multiplier rules for problems with a single \SVM, and then with multiple \SVM s are formulated.

The structure of the paper is as follows.
Sect.~\ref{S2} recalls some definitions and facts used throughout the paper.
The applicability of Theorem~\ref{T1.4} is illustrated in Sects.~\ref{S5}--\ref{S6} considering set-valued optimization problems with general preference relations.
A model with a single \SVM\ is studied in Sect.~\ref{S5}.
A particular case of this model
when the family $\Xi$ is determined
by an abstract level-set mapping is considered in Sect.~\ref{S4}.
A more general model with multiple \SVM s is briefly discussed in Sect.~\ref{S6}.

\section{Preliminaries}\label{S2}

Our basic notation is standard, see, e.g., \cite{Mor06.1,RocWet98,DonRoc14,Iof17}.
Throughout the paper, if not explicitly stated otherwise, $X$ and $Y$ are  normed spaces.
Products of normed spaces are assumed to be equipped with the maximum norm.
The topological dual of a normed space $X$ is denoted by $X^*$, while $\langle\cdot,\cdot\rangle$ denotes the bilinear form defining the pairing between the two spaces.
The open ball with center $x$ and radius $\delta>0$ is denoted by $B_\delta(x)$.
If $(x,y)\in X\times Y$, we write $B_\varepsilon(x,y)$ instead of $B_\varepsilon((x,y))$.
The open unit ball is denoted by $\B$ with a subscript indicating the space, e.g., $\B_X$ and $\B_{X^*}$.
Symbols $\R$ and $\N$ stand for the real line and the set of all positive integers, respectively.

The interior and closure of a set $\Omega$ are denoted by $\Int\Omega$ and $\cl\Omega$, respectively.
The distance from a point $x \in X$ to a subset $\Omega\subset X$ is defined by $d(x,\Omega):=\inf_{u \in \Omega}\|u-x\|$, and we use the convention $d(x,\emptyset)=+\infty$.

Given a subset $\Omega$ of a normed space $X$ and a point $\bx\in \Omega$, the sets (cf. \cite{Kru03,Cla83})
\begin{gather}\label{NC}
N_{\Omega}^F(\bx):= \Big\{x^\ast\in X^\ast\mid
\limsup_{\Omega\ni x{\rightarrow}\bar x,\;x\ne\bx} \frac {\langle x^\ast,x-\bx\rangle}
{\|x-\bx\|} \le 0 \Big\},
\\\label{NCC}
N_{\Omega}^C(\bx):= \left\{x^\ast\in X^\ast\mid
\ang{x^\ast,z}\le0
\qdtx{for all}
z\in T_{\Omega}^C(\bx)\right\}
\end{gather}
are the \emph{Fr\'echet} and \emph{Clarke normal cones} to $\Omega$ at $\bx$,
where $T_{\Omega}^C(\bx)$
stands for the \emph{Clarke tangent cone} to $\Omega$ at $\bx$:
\begin{multline*}
T_{\Omega}^C(\bx):= \big\{z\in X\mid
\forall x_k{\rightarrow}\bx,\;x_k\in\Omega,\;\forall t_k\downarrow0,\;\exists z_k\to z\\
\mbox{such that}\quad
x_k+t_kz_k\in \Omega \qdtx{for all}
k\in\N\big\}.
\end{multline*}
The sets \eqref{NC} and \eqref{NCC} are nonempty
closed convex cones satisfying $N_{\Omega}^F(\bx)\subset N_{\Omega}^C(\bx)$.
If $\Omega$ is a convex set, they reduce to the normal cone in the sense of convex analysis:
\begin{gather*}\label{CNC}
N_{\Omega}(\bx):= \left\{x^*\in X^*\mid \langle x^*,x-\bx \rangle \leq 0 \qdtx{for all} x\in \Omega\right\}.
\end{gather*}
By convention, we set $N_{\Omega}^F(\bx)=N_{\Omega}^C(\bx):=\es$ if $\bx\notin \Omega$.

A set-valued mapping $F:X\rightrightarrows Y$ between two sets $X$ and $Y$ is a mapping, which assigns to every $x\in X$ a (possibly empty) subset $F(x)$ of $Y$.
We use the notations $\gph F:=\{(x,y)\in X\times Y\mid
y\in F(x)\}$ and $\dom\: F:=\{x\in X\mid F(x)\ne\emptyset\}$
for the graph and the domain of $F$, respectively, and $F^{-1} : Y\rightrightarrows X$ for the inverse of $F$.
This inverse (which always exists with possibly empty values at some $y$) is defined by $F^{-1}(y) :=\{x\in X \mid y\in F(x)\}$, $y\in Y$. Obviously $\dom F^{-1}=F(X)$.

If $X$ and $Y$ are normed spaces, the \emph{Clarke coderivative} $D^{*C}F(x,y)$ of $F$ at $(x,y)\in\gph F$ is a set-valued mapping defined by
\begin{align}\label{coder}
D^{*C}F(x,y)(y^*):=\{x^*\in X^*\mid (x^*,-y^*)\in N^C_{\gph F}(x,y)\},
\quad
y^*\in Y^*.
\end{align}
Replacing the Clarke normal cone in \eqref{coder} by the \Fr\ one, we obtain the definition of the \emph{\Fr\ coderivative}.
\if{
For the case of a single-valued function $f:X\to Y$ between normed spaces, we simply write $D^{*C}f(x)(y^*)$ and $D^{*F}f(x)(y^*)$ for the Fr\'echet and Clarke coderivatives.
It follows directly from the definitions that
\begin{gather*}
D^{*F}f(\bx)(y^*)=\partial^F\langle y^*,f\rangle(\bx)
\;\;\text{for all}\;\; y^*\in Y^*
\end{gather*}
 if $f$ is  Lipschitz continuous at $\bx$, see, e.g., \cite[p. 609]{MorSha97.2} and \cite[Theorem~1.90]{Mor06.1}.
 Note that the latter result is for \textit{mixed} coderivatives, but its proof is also applicable to Fr\'echet coderivatives.
}\fi
\if{
\NDC{21.2.25
Should a proof for the above statement be provided?
I cannot find a similar statement for Clarke objects.
what do you think about the Clarke version of the above claim.}
\AK{22/02/25.
I don't understand.
It looks like you have commented out ``the claim''.}
}\fi

\begin{definition}
[Aubin property]
\label{D2.1}
A mapping $F:X\rightrightarrows Y$ between metric spaces has the Aubin property at $(\bx,\by)\in\gph F$
\red{with constant $\tau>0$ if there exists a $\de>0$}
such that
\begin{align}
\label{D2.10-1}
d(y,F(x))\le \tau d(x,x')
\qdtx{for all} x,x'\in B_{\de}(\bx),\;y\in F(x')\cap B_{\de}(\by).
\end{align}
\end{definition}

Aubin property
\red{(sometimes referred to as the locally \emph{Lipschitz-like} property)}
is among the most widely used properties of \SVM s in variational analysis (see, e.g., \cite{AubFra90,RocWet98,Mor06.1,DonRoc14,Iof17}).
It is known, in particular, to be equivalent to the \emph{metric regularity} of the inverse mapping.
It also yields estimates for the normals to the graph of the (given) mapping.

\begin{lemma}
\label{L2.3}
Let $X$ and $Y$ be normed spaces, $F:X\toto Y$, and $(\bx,\by)\in\gph F$.
\begin{enumerate}
\item
If $F$ has the Aubin property at $(\bx,\by)$ with \red{constant} $\tau>0$, then there is a $\de>0$ such that
\begin{gather}
\label{L2.3-1}
\|x^*\|\le\tau\|y^*\|
\qdtx{for all}(x,y)\in\gph F\cap B_\de(\bx,\by),\;(x^*,y^*)\in N^F_{\gph F}(x,y).
\end{gather}
\item
If $\dim Y<+\infty$, then $N^F$ in the above assertion can be replaced by $N^C$.
\end{enumerate}
\end{lemma}
\begin{proof}
\begin{enumerate}
\item
is well known; see, e.g., \cite[Theorem~1.43(i)]{Mor06.1}.
(The latter theorem is formulated in \cite{Mor06.1} in the Banach space setting, but the proof is valid in arbitrary normed spaces.)
\item
Suppose $\dim Y<+\infty$, and $F$ has the Aubin property at $(\bx,\by)$ with
\red{constant} $\tau>0$, i.e., condition \eqref{D2.10-1} is satisfied for some $\de>0$.
Let $(x,y)\in B_\de(\bx,\by)\cap\gph F$ and
$(x^*,y^*)\in N_{\gph F}^C(x,y)$.
Take any sequences $(x_k,y_k)\in\gph F$ and $t_k>0$ such that $(x_k,y_k)\to(x,y)$ and $t_k\downarrow0$ as $k\to+\infty$.
Fix an arbitrary $u\in X$.
Without loss of generality, we can assume that $x_k, x_k+t_ku\in B_\de(\bx)$ and $y_k\in B_\de(\by)$ for all $k\in\N$.
By \eqref{D2.10-1}, for each $k\in\N$, there exists a point $y_k'\in F(x_k+t_ku)$ such that $\|y_k'-y_k\|\le \tau t_k\|u\|$.
Set $v_k:=(y_k'-y_k)/t_k$.
Then $\|v_k\|\le\tau\|u\|$.
Passing to subsequences, we can suppose that $v_k\to v\in Y$.
Observe that $(u,v_k)\to (u,v)$ as $k\to+\infty$, and $(x_k,y_k)+t_k(u,v_k)\in\gph F$ for each $k\in\N$.
Thus, $(u,v)\in T_{\gph F}^C(x,y)$, and $\|v\|\le\tau\|u\|$.
By the definition of the
Clarke normal cone,
we have $\ang{x^*,u}\le-\ang{y^*,v}\le\tau\|y^*\|\|u\|$.
Since vector $u$ is arbitrary, it follows that $\|x^*\|\le\tau\|y^*\|$.
\qed\end{enumerate}
\end{proof}

\section{Set-Valued Optimization: A Single Mapping}
\label{S5}
Let $X$ and $Y$ be normed spaces, $\Omega\subset X$, $F:X\toto Y$, $\bx\in\Omega$ and $\by\in F(\bx)$.
To model the setting of Definition~\ref{D1.5}, we consider a nonempty family $\Xi$ of subsets of $Y$, and
\red{two families of subsets of $X\times Y$ given by \eqref{Xi}.}
To emphasize the structure of the pair \eqref{Xi}, when referring to the corresponding properties in Definition~\ref{D1.5}, we will talk about extremality/stationarity of the triple $\{F,\Omega,\Xi\}$.

\begin{definition}
\label{D4.1}
The triple $\{F,\Omega,\Xi\}$ is
extremal (resp., stationary, approximately stationary) at $(\bx,\by)$ if the pair
\eqref{Xi} is extremal (resp., stationary, approximately stationary) at~$(\bx,\by)$.
\end{definition}

The next proposition is a direct consequence of Definitions~\ref{D1.5} and \ref{D4.1}.

\begin{proposition}
\label{P3.3}
The triple $\{F,\Omega,\Xi\}$ is
\begin{enumerate}
\item
\label{P3.3.1}
extremal at $(\bx,\by)$ if and only if there is a $\rho\in(0,+\infty]$ such that, for any $\varepsilon>0$, there exists an $A\in\Xi$ such that
$d(\by,A)<\eps$, and
\begin{gather}
\label{P3.3-1}
F(\Omega\cap B_\rho(\bx))\cap A\cap B_\rho(\by)=\emptyset;
\end{gather}
\item
\label{P3.3.2}
stationary at $(\bx,\by)$ if and only if
for any $\varepsilon>0$, there exist a $\rho\in(0,\varepsilon)$ and an $A\in\Xi$ such that $d(\by,A)<\eps\rho$, and
condition \eqref{P3.3-1} is satisfied;
\item
\label{P3.3.3}
approximately stationary at $(\bx,\by)$ if and only if,
for any $\varepsilon>0$, there exist a $\rho\in(0,\varepsilon)$,
an $A\in\Xi$, and
$(x_1,y_1),(x_2,y_2)\in B_{\varepsilon}(\bx,\by)$
such that
$d((x_1,y_1),\gph F)<\eps\rho$, $d(x_2,\Omega)<\eps\rho$,
$d(y_2,A)<\eps\rho$, and
\begin{gather*}
F(x_1+(\Omega-x_2)\cap(\rho\B_X))\cap(y_1+( A-y_2)\cap (\rho\B_Y))=\emptyset.
\end{gather*}
\end{enumerate}
\end{proposition}
\red{
\begin{remark}
Definition~\ref{D4.1} gives rather general concepts of extremality/stationarity.
In the particular case when $F$ is single-valued and $\Xi:=\{K+F(\bx)-y\mid y\in Y\}$ for some subset $K\subset Y$ containing $0$, thanks to Proposition~\ref{P3.3}\,\eqref{P3.3.1}, the extremality in the sense of Definition~\ref{D4.1} means that
there is a $\rho\in(0,+\infty]$ and a sequence $\{y_k\}\subset Y$ such that
$d(y_k,K)\to0$ as $k\to+\infty$, and
\begin{gather*}
F(x)-F(\bx)\notin (K-y_k)\cap(\rho\B)
\;\;\text{for all}\;\; x\in\Omega\cap B_\rho(\bx)\;\;\text{and}\;\;k\in\N.
\end{gather*}
Clearly, $d(y_k,K)\to0$ if $y_k\to0$, in which case the above condition becomes a constrained (on $\Omega$) localized (in the image space) version of the $(F,K)$-optimality in \cite[Definition~5.53]{Mor06.2}.
As commented in \cite[p.~70]{Mor06.2}, when $K$ is a convex cone, the latter property covers the conventional notion of local Pareto optimality as well as local weak Pareto optimality if $\Int K\ne\es$.
\end{remark}
}

The next example illustrates relations between the properties in
\red{Definition~\ref{D4.1}}.



\begin{example}
\label{e4.2}
Let $X=Y=\Omega:=\R$, $\Xi:=\{(-\infty,t]\mid t\in\R\}$, and
$F_1,F_2,F_3,F_4:\R\toto \R$ be given by
\begin{align*}
F_1(x) &:= [0,+\infty)\quad \mbox{for all}\; x\in \R,
\quad\;
F_2(x) :=
\begin{cases}
[x+1,+\infty) &\mbox{if } x< -1,\\
[0,+\infty) &\mbox{if } x\ge-1,
\end{cases}
\\
F_3(x) &:=
[-x^2,+\infty)\; \mbox{for all}\; x\in \R,\quad
F_4(x) :=
\begin{cases}
[x,+\infty) &\mbox{if } x<0,\\
[-x^2,+\infty) &\mbox{if } x\ge0.
\end{cases}
\end{align*}
Then $0\in F_i(0)$ for all $i=1,2,3,4$.
The following assertions hold true:
\begin{enumerate}
\item
$\{F_1,\R,\Xi\}$ is extremal at $(0,0)$
with $\rho=+\infty$;
\item
$\{F_2,\R,\Xi\}$ is extremal at $(0,0)$ with some $\rho\in(0,+\infty)$ but not with $\rho=+\infty$;
\item
$\{F_3,\R,\Xi\}$ is stationary but not extremal at $(0,0)$;
\item
$\{F_4,\R,\Xi\}$ is approximately stationary at $(0,0)$ but not stationary at~$(0,0)$.
\end{enumerate}
The assertions are straightforward.
We only prove assertion (iv).
Let $\eps\in(0,1)$.
Choose any $\rho\in(0,\eps)$ and $t\in(\eps\rho,\rho)$.
Then $-t\in\rho\B_{\R}$, $-t\in F_4(-t)$ and
$-t\in A$ for any $A:=(-\infty,-\eta]\in\Xi$ with $d(0,A)<\eps\rho$ (i.e., for any $\eta<\eps\rho$).
By Proposition~\ref{P3.3}\,\eqref{P3.3.2}, $\{F_4,\R,\Xi\}$ is not stationary at $(0,0)$.

Let $\varepsilon>0$.
Choose a $\rho\in(0,\min\{\eps,1\}/3)$ and points
$(x_1,y_1):=(\rho,-\rho^2)\in\gph F_4\cap(\varepsilon\B_{\R^2})$,
$x_2:=0\in\varepsilon\B_{\R}$,
$y_2:=0\in\varepsilon\B_{\R}$.
Observe that $A:=(-\infty,-3\rho^2]\in\Xi$ satisfies $d(0,A)<\eps\rho$.
Then
\begin{gather*}
F_4(x_1+(\rho\B_{\R})) = F_4(0,2\rho) = (-4\rho^2,+\infty),\\
y_1+(A-y_2)\cap (\rho\B_{\R})=-\rho^2+(-\rho,-3\rho^2]=(-\rho^2-\rho,-4\rho^2],
\end{gather*}
and consequently,
$F_4(x_1+(\rho\B_{\R}))\cap (y_1+(A-y_2)\cap (\rho\B_{\R}))=\emptyset$.
By Proposition~\ref{P3.3}\,\eqref{P3.3.3}, $\{F_4,\R,\Xi\}$ is approximately stationary at $(0,0)$.
\end{example}

The following example shows that the family of sets $\Xi$ plays an important role in determining the properties.
\begin{example}
Let $X=Y=\Omega:=\R$ and
$F:\R\toto \R$ be given by
\begin{align*}
F(x):=
\begin{cases}
\{x\sqrt{2}\} &\mbox{if } x \text{ is rational},\\
\{x\} &\mbox{otherwise}.
\end{cases}
\end{align*}
Then $0\in F(0)$.

Let $\Xi:=\{(-\infty,t]\mid t\in\R\}$ and $\eps\in(0,1)$.
Choose any $\rho\in(0,\eps)$, $(x_1,y_1)\in\gph F$, $x_2\in\R$,
$y_2\in(-\varepsilon,0]$ and $A:=(-\infty,-t]\in\Xi$ with $d(y_2,A)<\eps\rho$ (i.e., $t<\eps\rho-y_2$).
Then
$x_1+(\Omega-x_2)\cap(\rho\B_X)=B_\rho(x_1)$,
$A-y_2=(-\infty,-\tau]$, where $\tau:=y_2+t<\eps\rho$,
and consequently,
$y_1+(A-y_2)\cap(\rho\B_Y)
\supset(y_1-\rho,y_1-\eps\rho)$.
We next show that $F(B_\rho(x_1))\cap(y_1-\rho,y_1-\eps\rho)\ne\es$.
If $x_1$ is rational, then
$y_1=x_1\sqrt{2}$, and choosing a rational number $\hat x\in(x_1-\rho/\sqrt{2},x_1-\eps\rho/\sqrt{2})\subset B_\rho(x_1)$, we get
$\hat y:=\hat x\sqrt{2}\in F(\hat x) \cap(y_1-\rho,y_1-\eps\rho)$.
If $x_1$ is irrational, then
$y_1=x_1$, and choosing an irrational number $\hat x\in(x_1-\rho,x_1-\eps\rho)\subset B_\rho(x_1)$, we get
$\hat y:=\hat x\in F(\hat x)\cap (y_1-\rho,y_1-\eps\rho)$.
By Proposition~\ref{P3.3}\,\eqref{P3.3.3}, $\{F,\R,\Xi\}$ is not approximately stationary at $(0,0)$.

Let $\Xi:=\{\{-1/n\}\mid n\in\N\}$.
Since $F(\R)$ only contains irrational numbers, we have
$F(\R)\cap A=\es$ for all $A\in\Xi$.
By Proposition~\ref{P3.3}\,\eqref{P3.3.1}, $\{F,\R,\Xi\}$ is extremal at $(0,0)$ (with $\rho=+\infty$).
\end{example}

Application of Theorem~\ref{T1.4} yields necessary conditions for approximate stationarity and, hence, also stationarity and extremality.

\begin{theorem}\label{T4.1}
Let $X$ and $Y$ be Banach spaces, the sets $\Omega$, $\gph F$ and all
{members of} $\Xi$ be closed.
If the triple $\{F,\Omega,\Xi\}$ is approximately stationary at $(\bx,\by)$, then, for any $\varepsilon>0$,
there exist $(x_1,y_1)\in\gph F\cap B_{\varepsilon}(\bx,\by)$,
$x_2\in\Omega\cap B_{\varepsilon}(\bx)$,
$A\in\Xi$,
$y_2\in A\cap B_\varepsilon(\by)$,
$(x_1^*,y_1^*)\in N^C_{\gph F}(x_1,y_1)$,
$x_2^*\in N^C_\Omega(x_2)$ and
$y_2^*\in N^C_{ A}(y_2)$ such that
\begin{gather*}
\|(x_1^*,y_1^*)+(x_2^*,y_2^*)\|<\eps
\qdtx{and}
\|(x_1^*,y_1^*)\|+\|(x_2^*,y_2^*)\|=1.
\end{gather*}

If $X$ and $Y$ are Asplund, then $N^C$ in the above assertion can be replaced by $N^F$.
\end{theorem}

\if{
\AK{6/08/24.
$K$ is not involved in the conclusions.
This does not seem right.
To be checked.}
\NDC{10/8/24.
I think the set $K$ appears in the `additional' condition of the dual vectors; this observation can be added as a remark.
Similar to Theorem~\ref{T1.4}, we can assume in Theorem~\ref{T4.1} that $\{F,\Omega,\Xi\}$ is approximately stationary at $(\bx,\by)$, i.e.,
for any $\varepsilon>0$, there exist a $\rho\in(0,\varepsilon)$, $(x_1,y_1),(x_2,y_2)\in B_{\varepsilon}(\bx,\by)$, and an $A\in\Xi_{\eps\rho}(y_2)$
satisfying
$\max\{d((x_1,y_1),\gph F),d(x_2,\Omega)\}<\varepsilon\rho$, and
\begin{gather*}
F(x_1+(\Omega-x_2)\cap(\rho\B_X))\cap(y_1+( A-y_2)\cap (\rho\B_Y))=\emptyset.
\end{gather*}
}
}\fi

The normalization condition $\|(x_1^*,y_1^*)\|+\|(x_2^*,y_2^*)\|=1$ in Theorem~\ref{T4.1} ensures that normal vectors $(x_1^*,y_1^*)$ to $\gph F$ remain sufficiently large when $\eps\downarrow0$, i.e., $x_1^*$ and $y_1^*$ cannot go to $0$ simultaneously.
The case when vectors $y_1^*$ are bounded away from $0$ (hence, one can assume $\|y_1^*\|=1$) is of special interest as it leads to a proper \emph{multiplier rule}.
A closer look at the alternative: either $y_1^*$ are bounded away from $0$ as $\eps\downarrow0$, or they are not (hence, vectors $x_1^*$ remain large), allows one to formulate the following consequence of Theorem~\ref{T4.1}.

\begin{corollary}
\label{C4.2}
Let $X$ and $Y$ be Banach spaces, the sets $\Omega$, $\gph F$ and all
members of $\Xi$ be closed.
If the triple $\{F,\Omega,\Xi\}$ is approximately stationary at $(\bx,\by)$, then one of the following assertions holds true:
\begin{enumerate}
\item
\label{C4.2.1}
there is an $M>0$ such that,
for any $\varepsilon>0$, there exist $(x_1,y_1)\in\gph F\cap B_{\varepsilon}(\bx,\by)$,
$x_2\in\Omega\cap B_{\varepsilon}(\bx)$, $ A\in\Xi$,
$y_2\in A\cap B_\varepsilon(\by)$,
and $y^*\in N^C_{ A}(y_2)+\eps\B_{Y^*}$
such that $\|y^*\|=1$ and
\begin{gather}
\label{C4.2-1}
0\in D^{*C}F(x_1,y_1)(y^*) +N^C_\Omega(x_2)\cap(M\B_{X^*})+\eps\B_{X^*};
\end{gather}
\item
\label{C4.2.2}
for any $\varepsilon>0$, there exist $(x_1,y_1)\in\gph F\cap B_{\varepsilon}(\bx,\by)$,
$x_2\in\Omega\cap B_{\varepsilon}(\bx)$,
$x_1^*\in D^{*C}F(x_1,y_1)(\eps\B_{Y^*})$ and $x_2^*\in N^C_\Omega(x_{2})$
such that
\sloppy
$\|x_1^*+x_2^*\|<\eps$
{and}
$\|x_1^*\|+\|x_2^*\|=1$.
\if
\NDC{20/2/24.
It looks a little bit strange to me that $\Xi$ does not involve in the statement of (ii).
I think it should be mentioned as a remark.}
\AK{21/01/24.
Part (ii) corresponds to `singular' behavior of $F$ (on $\Omega$): the case of `horizontal' normals to the graph; the $y^*$ component vanishes, and consequently, $K$ and $ A$ play no role.
Qualification conditions (e.g., Lipschitzness) are usually used to exclude singular behavior.}
\NDC{21/1/24.
Thank you for the explanation! I think Corollary~\ref{C4.2} is `similar' to \cite[Theorem~3.2]{ZheNg06} in the Asplund setting.}
\fi
\end{enumerate}

If $X$ and $Y$ are Asplund, then $N^C$ and $D^{*C}$ in the above assertions can be replaced by $N^F$ and $D^{*F}$, respectively.
\end{corollary}

\begin{proof}
Let the triple $\{F,\Omega,\Xi\}$ be approximately stationary at $(\bx,\by)$.
By Theorem~\ref{T4.1}, for any $j\in\N$, there exist $(x_{1j},y_{1j})\in\gph F\cap B_{1/j}(\bx,\by)$,
$x_{2j}\in\Omega\cap B_{1/j}(\bx)$,
${ A_j}\in\Xi$,
$y_{2j}\in A_j\cap B_{1/j}(\by)$,
$(x_{1j}^*,y_{1j}^*)\in N^C_{\gph F}(x_{1j},y_{1j})$,
$x_{2j}^*\in N^C_\Omega(x_{2j})$
and
$y^*_{2j}\in N^C_{ A_j}(y_{2j})$
such that $\|(x_{1j}^*,y_{1j}^*)\|+\|(x_{2j}^*,y_{2j}^*)\|=1$
and
$\|(x_{1j}^*,y_{1j}^*)+(x_{2j}^*,y_{2j}^*)\|<1/j$.
We consider two cases.
\sloppy

\emph{Case 1.}
$\limsup_{j\to+\infty}\|y_{1j}^*\|>\al>0$.
Note that $\al<1$.
Set $M:=1/\al$.
Let $\eps>0$.
Choose a number $j\in\N$ so that
$j\iv<\al\eps$
and $\|y_{1j}^*\|>\al$.
Set $x_1:=x_{1j}$, $y_1:=y_{1j}$, $x_2:=x_{2j}$, $ A:= A_j$, $y_2:=y_{2j}$, $y^*:=-y_{1j}^*/\|y_{1j}^*\|$, $x_1^*:=x_{1j}^*/\|y_{1j}^*\|$,
$x_2^*:=x_{2j}^*/\|y_{1j}^*\|$ and
$y_2^*:=y_{2j}^*/\|y_{1j}^*\|$.
Then $(x_1,y_1)\in\gph F\cap B_{\varepsilon}(\bx,\by)$,
$x_2\in\Omega\cap B_{\varepsilon}(\bx)$,
$y_2\in A\cap B_\varepsilon(\by)$,
$x_1^*\in D^{*C}F(x_1,y_1)(y^*)$,
$x_2^*\in N^C_\Omega(x_{2})$, ${\|y^*\|=1}$, $y^*_{2}\in N^C_{ A}(y_{2})$,
$\|x_2^*\|<1/\al=M$.
Furthermore,
$\|y^*-y_2^*\| =\|y_{1j}^*+y_{2j}^*\|/\|y_{1j}^*\|<1/(\al j)<\eps$, hence, $y^*\in N^C_{ A}(y_2)+\eps\B_{Y^*}$; and
$\|x_1^*+x_2^*\|=\|x_{1j}^*+x_{2j}^*\|/\|y_{1j}^*\|<1/(\al j)<\eps$, hence, condition \eqref{C4.2-1} is satisfied.
Thus, assertion \eqref{C4.2.1} holds true.

\emph{Case 2.}
$\lim_{j\to+\infty}\|y_{1j}^*\|=0$.
Then $y_{2j}^*\to0$, $x_{1j}^*+x_{2j}^*\to0$ and $1\ge\|x_{1j}^*\|+\|x_{2j}^*\|\to1$ as $j\to+\infty$.
Let $\eps>0$.
Choose a number $j\in\N$ so that
$\|x_{1j}^*\|+\|x_{2j}^*\|>0$ and $\max\{j\iv,\|y_{1j}^*\|, \|x_{1j}^*+x_{2j}^*\|/(\|x_{1j}^*\|+\|x_{2j}^*\|)\}<\eps$.
Set $x_1:=x_{1j}$, $y_1:=y_{1j}$, $x_2:=x_{2j}$, $x_1^*:=x_{1j}^*/(\|x_{1j}^*\|+\|x_{2j}^*\|)$ and $x_2^*:=x_{2j}^*/(\|x_{1j}^*\|+\|x_{2j}^*\|)$.
Then $(x_1,y_1)\in\gph F\cap B_{\varepsilon}(\bx,\by)$,
$x_2\in\Omega\cap B_{\varepsilon}(\bx)$,
$x_1^*\in D^{*C}F(x_1,y_1)(\eps\B_{Y^*})$,
$x_2^*\in N^C_\Omega(x_{2})$,
$\|x_1^*\|+\|x_2^*\|=1$,
and $\|x_1^*+x_2^*\|<\eps$.
Thus, assertion \eqref{C4.2.2} holds true.

If $X$ and $Y$ are Asplund, then
$N^C$ and $D^{*C}$ in
the above arguments can be replaced by
$N^F$ and $D^{*F}$, respectively.
\qed\end{proof}

\begin{remark}
\label{R4.3}
\begin{enumerate}
\item
Part \eqref{C4.2.1} of Corollary~\ref{C4.2} gives a kind of
\red{fuzzy}
multiplier rule with $y^*$ playing the role of the vector of multipliers.
\red{If $F$ is single-valued and Lipschitz continuous around $\bx$, then $D^{*F}(x_1,F(x_1))(y^*)=\partial^F\langle y^*,F\rangle(x_1)$ for all $y^*\in Y^*$ and all $x_1$ sufficiently close to $\bx$ (see, e.g., \cite[Theorem~1.90]{Mor06.1}).
If, additionally, $\Omega=X$, the Asplund space version of condition \eqref{C4.2-1} becomes $0\in\partial^F\langle y^*,F\rangle(x_1)+\eps\B_{X^*}$.
}

\item
\label{R4.3.2}
Part \eqref{C4.2.2} corresponds to `singular' behaviour of $F$ on $\Omega$.
It involves `horizontal' normals to the graph of $F$; the $y^*$ component vanishes, and consequently, $\Xi$ plays no role.
\end{enumerate}
\end{remark}

The following condition
is the negation of the condition in Corollary~\ref{C4.2}\,\eqref{C4.2.2}.
\begin{enumerate}
\item [ ]
\begin{enumerate}
\item [$(QC)_C$]
there is an $\varepsilon>0$ such that
$\|x_1^*+x_2^*\|\ge\eps$
for all
$(x_1,y_1)\in\gph F\cap B_{\varepsilon}(\bx,\by)$,
$x_2\in\Omega\cap B_{\varepsilon}(\bx)$,
$x_1^*\in D^{*C}F(x_1,y_1)(\eps\B_{Y^*})$ and
$x_2^*\in N^C_\Omega(x_{2}))$
such that $\|x_1^*\|+\|x_2^*\|=1$.
\end{enumerate}
\end{enumerate}
It excludes the singular behavior mentioned in Remark~\ref{R4.3}\,\eqref{R4.3.2} and serves as a qualification condition ensuring that only the condition in part \eqref{C4.2.1} of Corollary~\ref{C4.2} is possible.
We denote by $(QC)_F$ the analogue of $(QC)_C$ with $N^F$ and $D^{*F}$ in place of $N^C$ and $D^{*C}$, respectively.

\begin{corollary}
\label{C4.30}
Let $X$ and $Y$ be Banach spaces, $\Omega$, $\gph F$ and all members of $\Xi$ be closed.
Suppose that the triple $\{F,\Omega,\Xi\}$ is approximately stationary at $(\bx,\by)$.
If condition $(QC)_C$ is satisfied,
then assertion \eqref{C4.2.1} in Corollary~\ref{C4.2} holds true.

If $X$ and $Y$ are Asplund and condition $(QC)_F$ is satisfied,
then assertion \eqref{C4.2.1} in Corollary~\ref{C4.2} holds true with $N^F$ and $D^{*F}$ in place of $N^C$ and $D^{*C}$, respectively.
\end{corollary}

The next proposition provides two typical sufficient conditions for the fulfillment of conditions $(QC)_C$ and $(QC)_F$.

\begin{proposition}
\label{P4.11}
Let $X$ and $Y$ be
normed
spaces.
\begin{enumerate}
\item
\label{P3.2.1}
If $F$ has the Aubin property at $(\bx,\by)$, then $(QC)_F$ is satisfied.
If, additionally, $\dim Y<+\infty$, then $(QC)_C$ is satisfied too.
\item
\label{P3.2.2}
If $\bx\in\Int\Omega$, then both $(QC)_C$ and $(QC)_F$ are satisfied.
\end{enumerate}
\end{proposition}

\begin{proof}
\begin{enumerate}
\item
If $F$ has the Aubin property at $(\bx,\by)$, then, by Lemma~\ref{L2.3},
condition \eqref{L2.3-1} is satisfied with some $\tau>0$ and $\de>0$, and, if $\dim Y<+\infty$, then the latter condition is also satisfied with $N^C$ in place of $N^F$.
Hence, $(QC)_F$ is satisfied with $\eps:=1/(2\tau+1)$, as well as $(QC)_C$ if $\dim Y<+\infty$.
Indeed, if $(x_1,y_1)\in\gph F\cap B_\de(\bx,\by)$ and $x_1^*\in D^{*}F(x_1,y_1)(y_1^*)$ (where $D^{*}$ stands for either $D^{*C}$ or $D^{*F}$),
$x_2^*\in X^*$, $\|x_1^*\|+\|x_2^*\|=1$ and $\|y_1^*\|<\eps$, then
$\|x_1^*+x_2^*\|\ge\|x_2^*\|-\|x_1^*\|=1-2\|x_1^*\|> 1-2\tau\eps
=\eps.$
\item
If $\bx\in\Int\Omega$, then $N^C_\Omega(x_{2})=N^F_\Omega(x_{2})=\{0\}$ for all $x_{2}$ near $\bx$, and consequently,
for any normal vector $x_2^*$ to $\Omega$ at $x_2$ and any $x_1^*\in X^*$, condition $\|x_1^*\|+\|x_2^*\|=1$ yields $\|x_1^*+x_2^*\|=1$.
Hence, both $(QC)_C$ and $(QC)_F$ are satisfied with any sufficiently small~$\eps$.
\if
\NDC{6/3/24.
I think it suffices to assume in item (ii) that $X$ and $Y$ are normed spaces.}
\AK{6/3/24.
Also in item (i).
Shall we?}
\NDC{7/3/24.
I think we should.
}
\fi
\qed\end{enumerate}
\end{proof}

\red{As a consequence of Corollary~\ref{C4.2}, we obtain dual necessary conditions for the (local) Pareto optimality covering \cite[Theorems~3.1 and 4.1]{ZheNg05.2}, \cite[Corollary~3.1]{ZheNg06} and \cite[Proposition~5.1]{MorTreZhu03}.}

\begin{corollary}
\label{C3.4}
\red{Let $X$ and $Y$ be Banach spaces, $\Omega$ and $\gph F$ be closed.
If $(\bx,\by)$ is a (local) Pareto solution to  \eqref{P} with respect to a nontrivial pointed closed convex cone $K\subset Y$, then either
there is an $M>0$ such that,
for any $\varepsilon>0$, there exist $(x_1,y_1)\in\gph F\cap B_{\varepsilon}(\bx,\by)$,
$x_2\in\Omega\cap B_{\varepsilon}(\bx)$,
and $y^*\in Y^*$
such that $\langle y^*,y\rangle\ge0$ for all $y\in K$, $\|y^*\|=1$, and
\begin{gather*}
0\in D^{*C}F(x_1,y_1)(y^*+\varepsilon\B_{Y^*}) +N^C_\Omega(x_2)\cap(M\B_{X^*})+\eps\B_{X^*},
\end{gather*}
or assertion \eqref{C4.2.2} in Corollary~\ref{C4.2} holds true.
}

\red{If $X$ and $Y$ are Asplund, then $N^C$ and $D^{*C}$ in the above assertion can be replaced by $N^F$ and $D^{*F}$, respectively.}
\end{corollary}

\red{
\begin{proof}[Sketch]
Set $\Xi:=\{K+\by-y\mid y\in Y\}$ and observe that the conditions in Definition~\ref{Pareto} ensure that the triple $\{F,\Omega,\Xi\}$ is approximately stationary at $(\bx,\by)$.
Deducing the conclusion from Corollary~\ref{C4.2} requires straightforward renorming of the involved dual vectors.
\qed\end{proof}
}

\red{
\begin{remark}
In view of Proposition~\ref{P4.11}\;\eqref{P3.2.2}, if $\bx\in\Int\Omega$, then only the first alternative in Corollary~\ref{C3.4} is possible; cf. \cite[Theorems~3.1 and 4.1]{ZheNg05.2}.
\end{remark}	
}

\if{

\blue{
The next statement presents a dual characterization of the Aubin property in Asplund spaces; cf. \cite[Theorem~3.2]{Mor97}.
\begin{lemma}
Let $X$ and $Y$ be Asplund spaces, $F:X\toto Y$ with a closed graph, and $(\bx,\by)\in\gph F$.
The mapping $F$ has the Aubin property at $(\bx,\by)$ if and only if there exist $\tau>0$ and $\de>0$ such that
\begin{gather*}
\sup\{\|x^*\|\mid x^*\in D^{*F}F(x,y)(y^*)\}\le\tau\|y^*\|
\end{gather*}
for all $(x,y)\in\gph F\cap B_\de(\bx,\by)$ and $y^*\in Y^*$.
\end{lemma}
}
\blue{
\begin{corollary}
Let $X$ and $Y$ be Asplund spaces, $\Omega$, $\gph F$ and all members of $\Xi$ be closed.
Suppose $F$ has the Aubin property at $(\bx,\by)$.
If the triple $\{F,\Omega,K\}$ is approximately stationary at $(\bx,\by)$ with respect to $\Xi$,  then  for any $\varepsilon>0$, there exist $(x_1,y_1)\in\gph F\cap B_{\varepsilon}(\bx,\by)$,
$x_2\in\Omega\cap B_{\varepsilon}(\bx)$, $ A\in\Xi$,
$y_2\in A\cap(\eps\B_Y)$, and $y^*\in N^F_{ A}(y_2)+\eps\B^*_{Y^*}$
such that $\|y^*\|=1$ and
\begin{gather*}
0\in D^{*F}F(x_1,y_1)(y^*) +N^F_\Omega(x_2)+\eps\B_{X^*}.
\end{gather*}
\end{corollary}
}
\begin{proof}
Suppose $F$ has the Aubin property at $(\bx,\by)$.
By Lemma~\ref{L4.1}, there exist $\tau>0$ and $\tau>0$ such that condition \eqref{L4.1-1} holds
for all $(x,y)\in\gph F\cap B_\de(\bx,\by)$ and $y^*\in Y^*$.
We will show that condition \ref{C4.2.2} of Corollary~\ref{C4.2} does not hold with any $\varepsilon\in(0,\min\{\de,\tau\iv\})$.
Suppose that there exist
there exist $(x_1,y_1)\in\gph F\cap B_{\varepsilon}(\bx,\by)$,
$x_2\in\Omega\cap B_{\varepsilon}(\bx)$,
and $x^*\in D^{*F}F(x_1,y_1)(\eps\B_{Y^*})$
such that $\|x^*\|=1$ and
$d(-x^*,N^F_\Omega(x_{2}))<\varepsilon$.
Then, condition \eqref{L4.1-1} implies that
 $\|x^*\|\le\tau\varepsilon<1$, which is a contradiction.
\qed\end{proof}
}\fi

The next example illustrates the verification of the necessary conditions for approximate stationarity in Corollaries~\ref{C4.2} and \ref{C4.30} for the triple $\{F_4,\Omega,\Xi\}$, where $F_4$ is defined in Example~\ref{e4.2}.

\begin{example}
\label{e4.3}
Let $X=Y=\Omega:=\R$, and $F_4$ and $\Xi$ be as in Example \ref{e4.2}.
Thus, the triple $\{F_4,\Omega,\Xi\}$ is approximately stationary at $(\bx,\by)$, and the conclusions of Corollary~\ref{C4.2} must hold true.
Moreover, the assumptions in both parts of Proposition~\ref{P4.11} are satisfied, and consequently, condition $(QC)_F$ holds true.
By Corollary~\ref{C4.30}, assertion \eqref{C4.2.1} in Corollary~\ref{C4.2} holds true with $N^F$ and $D^{*F}$ in place of $N^C$ and $D^{*C}$, respectively.
\if{
We first check condition (QC$_F$).
It is easy to see that $F_4$ has the Aubin property at every point in its graph, and one can employ Remark~\ref{R4.4}\,\ref{R4.4.1}, but we check condition (QC$_F$) directly.
For that we need to compute \Fr\ normal cones to the graph of $F_4$ at all its points near $(0,0)$.
If $(x,y)\in\Int\gph F_4$ or $(x,y)=(0,0)$, then obviously $N^F_{\gph F_4}(x,y)=\{(0,0)\}$.
If $x<0$, then $N^F_{\gph F_4}(x,x)=\{(t,-t)\mid t\ge0\}$.
If $x>0$, then $N^F_{\gph F_4}(x,-x^2)=\{(-2tx,-t)\mid t\ge0\}$.
Hence, if $(x,y)\in\gph F_4\cap (\frac12\B_{\R^2})$ and $x^*\in D^{*F}F_4(x,y)(\frac12\B_{\R})$, then $|x^*|<\frac12$, and consequently, $D^{*F}F_4(x,y)(\frac12\B_{\R})\cap\Sp_{\R}=\es$.
This verifies condition (QC$_F$).
In view of Remark~\ref{R4.3}\,\ref{R4.3.2}, assertion (\ref{C4.2.1}) of Corollary \ref{C4.2} must hold true.
}\fi
We now verify this assertion.

Let $M>0$ and $\varepsilon>0$.
Choose a $t\in(0,\min\{\eps,1\})$.
Set $(x_1,y_1):=(t/2,-t^2/4)\in\gph F_4\cap(\varepsilon\B_{\R^2})$,
$x_2:=0\in\Omega\cap{(\eps\B_{\R})}$,
$ A:=(-\infty,-t]\in\Xi$,
$y_2:=-t\in A\cap{(\eps\B_{\R})}$ and $y^*:=1$.
Thus, $N^F_\Omega(x_2)=\{0\}$, $y^*\in N^F_{ A}(y_2)$ and $D^{*F}F_4(x_1,y_1)(y^*) =\{-t\}$.
Hence,
\begin{gather*}
D^{*F}F_4(x_1,y_1)(y^*)+N^F_\Omega(x_2)\cap(M\B_{\R}) =\{-t\}\in\eps\B_{\R},
\end{gather*}
i.e., assertion \eqref{C4.2.1} in Corollary \ref{C4.2} holds true (in terms of \Fr\ normals and coderivatives).
\if
\AK{13/01/24 and 24/01/24.
This example is good for testing our own understanding of our theory, but is of little value for the paper as it is handled within the conventional setting of linear translations.
}
\NDC{14/1/24,
I think it has been shown  in \cite[Corollary~3.2]{ZheNg06} (in the setting of Asplund spaces and Pareto relations) that condition (ii) cannot happen when both $F$ and $G$ have the Aubin property.
}
\AK{15/01/24.
What is $G$?
Regularity of $\{F,\Omega\}$ is needed.
I think you can find a kind of ``normal qualification condition'' in Mordukhovich.
}
\NDC{16/1/24.
I just realized that I had cited a wrong preference.
It is fixed now.
I meant item (ii) of Corollary~\ref{C4.4} (if it is true!) cannot happen if $F$ and $G$ have the Aubin property thanks to \cite[Corollary~3.2]{ZheNg06}.
Of course, everything needs to be proved within the current framework, i.e. with the general prefernce.
}
\fi
\end{example}

\section{Abstract Level-Set Mapping}
\label{S4}

We now consider a particular case of the model in Sections~\ref{S5}
when the family $\Xi$ is determined
by an abstract
\emph{level-set mapping} $L:Y\toto Y$.
The latter mapping defines a \emph{preference} relation $\prec$ on $Y$: $v\prec y$ if and only if $v\in L(y)$; see, e.g., \cite[p.~67]{KhaTamZal15}.

Given a point $y\in Y$, we employ below the following notations:
\begin{gather}
\label{L}
L^\circ(y):=L(y)\setminus\{y\},\quad
L^-(y):=L(y)\cup\{y\}.
\end{gather}

Certain requirements are usually imposed on $L$ in order to make the corresponding preference relation meaningful
in optimization and applications; see, e.g., \cite{Zhu00,MorTreZhu03,Mor06.2,KhaTamZal15}.
In this section, we discuss the following properties of $L$ at or near the reference point $\by$:
\renewcommand\theenumi{\rm O\arabic{enumi}}
\begin{enumerate}
\item
\label{O1}
$\liminf\limits_{L^\circ(\by)\ni y\to\by} d(\by,L(y))=0$;
\item
\label{O2}
$\by\in\cl L^\circ(\by)$;
\item
\label{O3}
$\by\notin L(\by)$;
\item
\label{O4}
$y\in\cl L(y)$ for all $y$ near $\by$;
\item
\label{O5}
if $y\in L^\circ(\by)$ and $v\in\cl L(y)$, then $v\in L^\circ(\by)$;
\item
\label{O6}
if $y\in L(\by)$ and $v\in\cl L(y)$, then $v\in L(\by)$.
\end{enumerate}
\renewcommand {\theenumi} {\rm\roman{enumi}}

Some characterizations of the properties and relations between them are collected in the next proposition.

\begin{proposition}
\label{P4.10}
Let $L:Y\toto Y$, $L^\circ$ be given by \eqref{L},
\red{and $\by\in Y$}.
The following assertions hold true.
\begin{enumerate}
\item
\label{P4.10.1}
\eqref{O1} \iff
$\{y\in Y\mid d(\by,L(y))<\eps\}\cap L^\circ(\by)\cap B_\eps(\by)\ne\es$ for all $\eps>0$.
\item
\eqref{O1} \folgt \eqref{O2}.
\item
\eqref{O3} \iff
$[L(\by)=L^\circ(\by)]$.
\item
\eqref{O2} \& \eqref{O4} \folgt \eqref{O1}.
\item
\eqref{O3} \& \eqref{O4} \folgt \eqref{O2}.
\item
\eqref{O3} \folgt $[\eqref{O5} \Leftrightarrow \eqref{O6}]$.
\end{enumerate}
\end{proposition}

\begin{proof}
\begin{enumerate}
\item
\eqref{O1} \quad\iff\quad
$\inf_{y\in L^\circ(\by)\cap B_\eps(\by)} d(\by,L(y))=0$ for any $\eps>0$ \quad\folgt\quad
for any $\eps>0$, there is a $y\in L^\circ(\by)\cap B_\eps(\by)$ such that $d(\by,L(y))<\eps$.
This proves the `$\Rightarrow$' implication.
Conversely, let $\de:=\inf_{y\in L^\circ(\by)\cap B_\eps(\by)} d(\by,L(y))>0$ for some $\eps>0$.
Then $\{y\in Y\mid d(\by,L(y))<\de\}\cap L^\circ(\by)\cap B_\eps(\by)=\es$, and consequently,
$\{y\in Y\mid d(\by,L(y))<\eps'\}\cap L^\circ(\by)\cap B_{\eps'}(\by)=\es$, where $\eps':=\min\{\eps,\de\}$.
The implication `$\Leftarrow$' follows.
\item
\eqref{O1} \quad\folgt\quad
there exists a sequence $\{y_k\}\subset L^\circ(\by)$ with $y_k\to\by$ \quad\iff\quad \eqref{O2}.
\item
The assertion is a consequence of the definition of $L^\circ$ in \eqref{L}.
\item
Suppose conditions \eqref{O2} and \eqref{O4} are satisfied.
Let $\eps>0$.
Thanks to \eqref{O4}, we can
choose a $\xi\in(0,\eps)$ such that $y\in\cl L(y)$ for all $y\in B_\xi(\by)$.
If $y\in B_\xi(\by)$,
then $d(\by,L(y))=d(\by,\cl L(y))\le\|y-\by\|<\xi$.
Thus, $B_\xi(\by)\subset \{y\in Y\mid d(\by,L(y))<\xi\}$.
Thanks to \eqref{O2}, we have
$$\{y\in Y\mid d(\by,L(y))<\eps\}\cap L^\circ(\by)\cap B_\eps(\by)\supset L^\circ(\by)\cap B_\xi(\by)\ne\es.$$
Since $\eps$ is an arbitrary positive number, in view of \eqref{P4.10.1}, this proves \eqref{O1}.
\item
\eqref{O4} \quad\folgt\quad
$\by\in\cl L(\by)$.
The conclusion follows thanks to (iii).
\item
The assertion is a consequence of (iii).
\qed\end{enumerate}
\end{proof}

\begin{remark}
Properties \eqref{O4} and \eqref{O6} are components of the definition of \emph{closed preference} relation (see \cite[Definition~5.55]{Mor06.2}, \cite[p.~583]{Bao14.2}, \cite[p.~68]{KhaTamZal15}) widely used in vector and set-valued optimization.
They are called, respectively, \emph{local satiation} (around $\by$) and \emph{almost transitivity}.
Note that the latter property is actually stronger than the conventional transitivity.
It is not satisfied for the preference defined by the \emph{lexicographical order} (see \cite[Example~5.57]{Mor06.2}) and some other natural preference relations important in vector optimization and its
applications including those to welfare economics (see \cite[Sect. 15.3]{KhaTamZal15}).
Closed preference relations are additionally assumed in \cite{Mor06.2,Bao14.2,KhaTamZal15} to be \emph{nonreflexive}, thus, satisfying, in particular, property \eqref{O3}.
In view of Proposition~\ref{P4.10}, if a preference relation satisfies properties \eqref{O3}, \eqref{O4} and \eqref{O6}, it also satisfies properties \eqref{O1}, \eqref{O2} and \eqref{O5}.
In this section, we employ the weaker properties \eqref{O1} and \eqref{O5}, which are satisfied by most conventional and many other preference relations.
This makes our model applicable to a wider range of multiobjective and set-valued optimization problems compared to those studied in \cite{Mor06.2,Bao14.2,KhaTamZal15}.
\end{remark}
\if
\AK{29/01/24.
The above consideration shows the importance of property \eqref{O1} in this topic.
It looks like in \cite{MorTreZhu03} and the subsequent publications they employ \eqref{O4} together with the nonreflexivity (everywhere!) assumption to ensure \eqref{O1}.
Without the nonreflexivity, property \eqref{O2} is needed as a replacement of \eqref{O1}.
This seems to be exactly what Thao picked up some time ago when he required $\by$ to be ``not an isolated point'' in his ``Proposition 5$\_$1 amended'' note.
}
\fi

\if{
\begin{example}
Suppose $\prec$ is the preference relation in $\R^2$ defined by the \emph{lexicographical} order, i.e., $(u_1,u_2)\prec (v_1,v_2)$ if and only if either $u_1<v_1$, or $u_1=v_1$ and $u_2<v_2$.
For any $y=(y_1,y_2)\in\R^2$ and $\eps>0$, we have
$L(y)=L^\circ(y)=((-\infty,y_1)\times \R) \cup (\{y_1\} \times (-\infty,y_2))$,
and
$L^\bullet_\eps(y) =(y_1-\eps,+\infty)\times \R$.
Note that
$y\notin L(y)$.
\end{example}
}\fi
\if{
\NDC{6/1/24.
Let $\prec$ be a preference relation on $\R^2$ define by the lexicographicial order, i.e. $u:=(u_1,u_2)\prec v:=(v_1,v_2)$ if and only if `$u_1<v_1$' or `$u_1=v_1$ and
$u_2<v_2$', and let $y_0:=(1,2)\in\R^2$.
Then
$$
L(y_0)=\{(a,b)\in\R^2\mid a<1,b\in\R \}\cup\{(1,c)\in\R^2\mid c<2 \}.$$
Observe that $y_0\notin L(y_0)$.
Thus, $L^\circ(y_0)=L(y_0)$.
One has
$$
\text{cl}L(y_0)=\{(a,b)\in\R^2\mid a\le1,b\in\R \}\cup\{(1,c)\in\R^2\mid c\le2\}.$$
Thus,
$N_{\cl L(y_0)}(1,b)=\R_+\times\{0\}$
for $b>2$,
$N_{\cl L(y_0)}(a,2)=\{0\}\times \R_+$ for $a>1$, and
$N^C_{\cl L(y_0)}(1,2)=N^F_{\cl L(y_0)}(1,2)=\{(0,0)\}$.}
\AK{6/01/24.
Have I misinterpreted your example?
What does it suppose to say?}
\NDC{6/1/24.
Thank you for correcting my mistakes.
Let me think about the second question!}
\Thao{7/1/2024. I think examples are not needed at this stage because the required properties of the preference, and hence the three set-valued mappings, have not been specified. Then, any binary relation will yield all the mappings, but without the properties of interest.}
}\fi

\red{The next proposition addresses some reasonably conventional settings.
\begin{proposition}
\label{P4.02}
Let $L(y):=y-K$ for some $K\subset Y$ and all $y\in Y$.
Let $\by\in Y$.
Denote $K^\circ:=K\setminus\{0\}$, $K^-:=K\cup\{0\}$.
Suppose that $0\in\cl K^\circ$.
Then
\begin{enumerate}
\item
\label{P4.02.1}
$L^\circ(y)=y-K^\circ$ and $L^-(y)=y-K^-$ for all $y\in Y$;
\item
\label{P4.02.2}
properties \eqref{O1}, \eqref{O2} and \eqref{O4} are satisfied;
\item
\label{P4.02.3}
if $0\notin K$, then property \eqref{O3} is satisfied;
\item
\label{P4.02.4}
if $K$ is an open convex cone and $K\ne Y$, then $K^\circ=K$ and properties \eqref{O5} and \eqref{O6} are  satisfied;
\item
\label{P4.02.5}
if $K$ is a closed convex cone, then $K^-=K$ and property \eqref{O6} is  satisfied.
\end{enumerate}
\end{proposition}
}
\red{
\begin{proof}
\begin{enumerate}
\item
is obvious.
\item
Let $y_k\in L^\circ(\by)$ $(k\in\N)$ and $y_k\to\by$ as $k\to+\infty$.
By \eqref{P4.02.1}, $c_k:=\by-y_k\in K$ $(k\in\N)$ and $c_k\to0$ as $k\to+\infty$.
Then $v_k:=y_k-c_k\in L(y_k)$ $(k\in\N)$ and $v_k\to\by$ as $k\to+\infty$.
This proves \eqref{O1}.
Since $0\in\cl K^\circ$, it follows from \eqref{P4.02.1} that $\by\in\cl L^\circ(\by)$.
This proves \eqref{O2}.
By the assumption, $0\in\cl K$, and consequently, $y\in y-\cl K=\cl L(y)$ for all $y\in Y$.
Property \eqref{O4} follows.
\item
is a consequence of \eqref{P4.02.1}.
\item
Let $K$ be an open convex cone and $K\ne Y$.
Then $0\notin K$, and consequently, $K^\circ=K$.
If $y\in L^\circ(\by)$ and $v\in\cl L(y)$, then $\by-y\in K$ and $y-v\in\cl K$; hence, $\by-v\in K+\cl K=K=K^\circ$, i.e.,
$v\in L^\circ(\by)=L(\by)$.
\item
Let $K$ be a closed convex cone.
Then $0\in K$, and consequently, $K^-=K$.
If $y\in L(\by)$ and $v\in\cl L(y)$, then $\by-y\in K$ and $y-v\in K$; hence, $\by-v\in K+K=K$, i.e., $v\in L(\by)$.
\qed\end{enumerate}		
\end{proof}
}
\red{
\begin{corollary}
Let $K$ be a nontrivial open convex cone, and $L(y):=y-K$ for all $y\in Y$.
Let $\by\in Y$.
Then properties \eqref{O1}--\eqref{O6} are  satisfied.
\end{corollary}
}

The next two examples illustrate some characterizations of the level-set mapping.

\begin{example}\label{E4.3}
Let $L(y):=\{y\}$ for all $y\in Y$.
Then $L^\circ(y)=\es$.
Thus, properties \eqref{O4} and \eqref{O5} are obviously satisfied, while properties \eqref{O1} and \eqref{O2} are violated.
\end{example}

\begin{example}
\label{E4.4}
Let $L:\R^2\toto\R^2$ be defined by
\begin{gather*}
L(y_1,y_2):=
\begin{cases}
\{(v_1,v_2)\in\R^2\mid v_1<y_1,\;v_2<y_2\}&\text{if } (y_1,y_2)\ne(0,0),
\\
\{(0,0)\}&\text{otherwise}.
\end{cases}
\end{gather*}
Let $\by:=(0,0)$.
Then $L^\circ(y_1,y_2)=L(y_1,y_2)$ if $(y_1,y_2)\ne\by$ and $L^\circ(\by)=\es$.
As in Example~\ref{E4.3}, properties \eqref{O4} and \eqref{O5} are satisfied, while properties \eqref{O1} and \eqref{O2} are violated.
\end{example}

\red{Given a level-set mapping $L$, a point $\by\in Y$ and a number $\de>0$,}
we are going to employ in our model the `localized' family of sets
\begin{gather}
\label{KXi}
\Xi^\de:=\{\cl L(y)\mid y\in L^-(\by)\cap B_\de(\by)\}.
\end{gather}
Note that
members of $\Xi^\de$ are not simply translations (deformations) of the fixed set $L(\by)$ (or $L^\circ(\by)$); they are defined by sets $L(y)$ where $y$ does not have to be equal to $\by$.

\begin{remark}
Given a set $K$ containing
\red{$0$}
one can naturally define the level-set mapping by $L(y)=y
+K$ for all $y\in Y$.
Then \eqref{KXi} defines the family of perturbations as the traditional collection of translations of $\cl K$, i.e., $\Xi^\de=
\red{\{y+\by+\cl K\mid y\in K\cap(\de B)}\}
$.
\end{remark}

In the current setting, the properties in Proposition~\ref{P3.3} take the following form.

\begin{proposition}
\label{P4.1}
Let $\de>0$, and $\Xi^\de$ be given by \eqref{KXi}.
The triple $\{F,\Omega,\Xi^\de\}$ is
\begin{enumerate}
\item
\label{P4.1.2}
extremal at $(\bx,\by)$ if and only if there is a $\rho\in(0,+\infty]$ such that, for any $\varepsilon>0$, there exists a $y\in L^-(\by)\cap B_\de(\by)$ such that
$d(\by,L(y))<\eps$, and
\begin{gather}
\label{P3.14-1}
F(\Omega\cap B_\rho(\bx))\cap\cl L(y)\cap B_\rho(\by)=\emptyset;
\end{gather}
\item
stationary at $(\bx,\by)$
if and only if, for any $\varepsilon>0$, there exist a $\rho\in(0,\varepsilon)$ and a $y\in L^-(\by)\cap B_\de(\by)$ such that
$d(\by,L(y))<\eps\rho$, and condition \eqref{P3.14-1} is satisfied;

\item
approximately stationary at $(\bx,\by)$ if and only if,
for any $\varepsilon>0$, there exist a $\rho\in(0,\varepsilon)$,
a $y\in L^-(\by)\cap B_\de(\by)$, and
$(x_1,y_1),(x_2,y_2)\in B_{\varepsilon}(\bx,\by)$
such that
$d((x_1,y_1),\gph F)<\eps\rho$, $d(x_2,\Omega)<\eps\rho$,
$d(y_2,L(y))<\eps\rho$, and
\begin{gather*}
F(x_1+(\Omega-x_2)\cap(\rho\B_X))\cap(y_1+(\cl L(y)-y_2)\cap(\rho\B_Y))=\emptyset.
\end{gather*}
\end{enumerate}
\end{proposition}

The statements of Theorem~\ref{T4.1} and its corollaries can be easily adjusted to the current setting.
\if{
For instance, Corollary~\ref{C4.1} can be reformulated as follows.

\begin{corollary}
\label{C4.3}
Let $X$ and $Y$ be Banach spaces, $\Omega$ and $\gph F$ be closed, $\de>0$, and $K$ and $\Xi^\de$ be given by \eqref{KXi}.
\if{
\NDC{6/1/24.
Quite many `and'!
}
\AK{6/01/24.
True. Not sure how this issue can be fixed: the objects come in pairs.}
\NDC{6/1/24.
Now we have only three `and'.
I think it is better now.}
}\fi
If the triple $\{F,\Omega,K\}$ is approximately stationary at $(\bx,\by)$ with respect to $\Xi$,  then, for any $\varepsilon>0$, there exist $(x_1,y_1)\in\gph F\cap B_{\varepsilon}(\bx,\by)$,
$x_2\in\Omega\cap B_{\varepsilon}(\bx)$,
$y\in K\cap B_\de(\by)$,
$y_2\in\cl L(y)\cap B_\eps(\by)$, $x^*\in N^C_\Omega(x_2)$,
and
$y^*\in N^C_{\cl L(y)}(y_2)$
such that $\|(x^*,y^*)\|=1$ and
\sloppy
\begin{gather*}
d\big(-(x^*,y^*),N_{\gph F}^{C}(x_1,y_1)\big)<\eps.
\end{gather*}
If $X$ is Asplund, then $N^C$ in the above assertion can be replaced by $N^F$.
\end{corollary}

For instance, Corollary~\ref{C4.2} can be reformulated as follows.

\begin{corollary}
Let $X$ and $Y$ be Banach spaces, $\Omega$ and $\gph F$ be closed, $\de>0$, and $K$ and $\Xi^\de$ be given by \eqref{KXi}.
If the triple $\{F,\Omega,K\}$ is approximately stationary at $(\bx,\by)$ with respect to $\Xi$,  then one of the following assertions holds true:
\begin{enumerate}
\item
there is an $M>0$ such that,
for any $\varepsilon>0$, there exist $(x_1,y_1)\in\gph F\cap B_{\varepsilon}(\bx,\by)$,
$x_2\in\Omega\cap B_{\varepsilon}(\bx)$, $y\in K\cap B_\de(\by)$,
$y_2\in\cl L(y)\cap B_\eps(\by)$,
and $y^*\in N^C_{\cl L(y)}(y_2)+\eps\B_{Y^*}$
such that $\|y^*\|=1$ and condition \eqref{C4.2-1} holds true;
\item
for any $\varepsilon>0$, there exist $(x_1,y_1)\in\gph F\cap B_{\varepsilon}(\bx,\by)$,
$x_2\in\Omega\cap B_{\varepsilon}(\bx)$,
and $x^*\in D^{*C}F(x_1,y_1)(\eps\B_{Y^*})$
such that $\|x^*\|=1$ and
$d(-x^*,N^C_\Omega(x_{2}))<\eps$.
\end{enumerate}
If $X$ is Asplund, then $N^C$ and $D^{*C}$ in the above assertions can be replaced by $N^F$ and $D^{*F}$, respectively.
\end{corollary}
}\fi
For instance, Corollary~\ref{C4.30} can be reformulated as follows.

\begin{corollary}
\label{C4.40}
Let $X$ and $Y$ be Banach spaces, $\Omega$ and $\gph F$ be closed, $\bx\in\Omega$, $\by\in F(\bx)$, $\de>0$, and $\Xi^\de$ be given by \eqref{KXi}.
Suppose condition $(QC)_C$ is satisfied.
If the triple $\{F,\Omega,\Xi^\de\}$ is approximately stationary at $(\bx,\by)$,
then
there is an $M>0$ such that,
for any $\varepsilon>0$, there exist $(x_1,y_1)\in\gph F\cap B_{\varepsilon}(\bx,\by)$,
$x_2\in\Omega\cap B_{\varepsilon}(\bx)$,
$y\in L^-(\by)\cap B_\de(\by)$,
$y_2\in\cl L(y)\cap B_\eps(\by)$,
and $y^*\in N^C_{\cl L(y)}(y_2)+\eps\B_{Y^*}$
such that $\|y^*\|=1$, and condition \eqref{C4.2-1} holds true.
\sloppy

If $X$ is Asplund and condition $(QC)_F$ is satisfied,
then the above assertion holds true with $N^F$ and $D^{*F}$ in place of $N^C$ and $D^{*C}$, respectively.
\end{corollary}

The
properties in Definition~\ref{D4.1} are rather general.
They cover various
optimality and stationarity concepts in vector and set-valued optimization.
With $\Omega$, $F$ and $L$ as above, and points $\bx\in\Omega$ and $\by\in F(\bx)$, the next definition seems reasonable.

\begin{definition}
\label{D4.2}
The point $(\bx,\by)$ is
\emph{extremal} for
$F$ on $\Omega$ if there is a $\rho\in(0,+\infty]$
such that
\begin{align}
\label{LE}
F(\Omega\cap B_\rho(\bx))\cap L^\circ(\by)\cap B_\rho(\by)=\es.
\end{align}
\end{definition}

Definition~\ref{D4.2} covers both local ($\rho<+\infty$) and global ($\rho=+\infty$) extremality.
The above concept is applicable, in particular, to
solutions of the set-valued minimization problem~\eqref{P},
\red{and the conventional Pareto optimality implies the extremality in the sense of Definition~\ref{D4.2}.
Indeed,
if $(\bx,\by)$ is a (local) Pareto solution to \eqref{P} with respect to a nontrivial pointed convex cone $K\subset Y$, then, by Definition~\ref{Pareto}, there is a $\rho\in(0,+\infty]$
such that $F(\Omega\cap B_\rho(\bx))\cap L^\circ(\by)=\es$, where $L^\circ(\by):=(\by-K)\setminus\{\by\}$.
The latter condition obviously implies condition \eqref{LE}.
Hence, $(\bx,\by)$ is an extremal point for
$F$ on $\Omega$.}

\begin{remark}
\label{R4_local_EX}
\begin{enumerate}
\item
The concept in Definition~\ref{D4.2} is broader than just (local) minimality as $F$ is not assumed to be an objective mapping of an optimization problem. It can, for instance, be involved in modeling constraints.
\item
The property in Definition~\ref{D4.2} is similar to the one in the definition of \emph{fully localized minimizer} in \cite[Definition~3.1]{BaoMor10} (see also \cite[p. 68]{KhaTamZal15}).
The latter definition uses the larger set $\cl L(\by)\setminus\{\by\}$
in place of $L^\circ(\by)$ in \eqref{LE}.
Unlike many solution concepts in vector optimization, the above definition involves ``image localization'' (hence, is in general weaker).
It has proved to be useful when studying locally optimal allocations of welfare economics; cf. \cite{BaoMor10,KhaTamZal15}.
\end{enumerate}
\end{remark}
\if
\Thao{30/12/2023. Should it be `locally optimal' instead of `locally extremal'?}
\AK{30/12/23.
Possibly.
My original reasoning is in the comment below.
Problem \eqref{P} was added later.}
\fi
\if
\AK{24/01/24.
After having a look at \cite[p. 68]{KhaTamZal15}, I am now thinking about restoring $L^\circ(\by)$ instead of $\by+L^\circ(0)$ and working with ``fully localized minimizers''.}
\AK{28/11/23.
`Local extremality' in the above definition is broader than just local minimality in the above problem.
Besides the objective mapping, it applies also to constraint mappings.
}
\fi

We next show that, under some mild assumptions on the level-set mapping $L$, the extremality in the sense of Definition~\ref{D4.2}
can be treated in the framework of the extremality in the sense of Definition~\ref{D4.1} (or its characterization in Proposition~\ref{P4.1}\,\eqref{P4.1.2}).

\begin{proposition}
\label{P4.2}
Let $\bx\in\Omega$, $\by\in F(\bx)$, $\de>0$, and $\Xi^\de$ be given by \eqref{KXi}.
Suppose $L$ satisfies conditions \eqref{O1} and \eqref{O5}.
If $(\bx,\by)$ is extremal for
$F$ on $\Omega$, then the triple $\{F,\Omega,\Xi^\de\}$ is extremal at $(\bx,\by)$.
\end{proposition}

\begin{proof}
In view of \eqref{O1}, it follows from Proposition~\ref{P4.10}\,\eqref{P4.10.1} that
\begin{gather}
\label{P4.2P1}
\{y\in Y\mid d(\by,L(y))<\eps\}\cap L^\circ(\by)\cap B_{\eps}(\by)\ne\es
\;\;\text{for all}\;\;
\eps>0.
\end{gather}
Suppose $\{F,\Omega,\Xi^\de\}$ is not extremal at $(\bx,\by)$.
Let $\rho\in(0,+\infty]$.
By Proposition~\ref{P4.1}\,\eqref{P4.1.2}, there exists an $\eps>0$ such that, for any $y\in L^-(\by)\cap B_\de(\by)$ with
$d(\by,L(y))<\eps$, it holds
\begin{gather}
\label{P4.2P2}
F(\Omega\cap B_\rho(\bx))\cap\cl L(y)\cap B_\rho(\by)\ne\es.
\end{gather}
In view of \eqref{P4.2P1}, there is a point $y\in L^\circ(\by)\cap B_\de(\by)\subset L^-(\by)\cap B_\de(\by)$ with
$d(\by,L(y))<\eps$, and we can choose a point
$\hat y$ belonging to the set in \eqref{P4.2P2}.
Thus, $y\in L^\circ(\by)$ and $\hat y\in\cl L(y)$.
Thanks to \eqref{O5}, we have $\hat y\in L^\circ(\by)$, and consequently, $\hat y\in F(\Omega\cap B_\rho(\bx))\cap L^\circ(\by)\cap B_\rho(\by)$.
Since $\rho\in(0,+\infty]$ is arbitrary, $(\bx,\by)$ is not
extremal for
$F$ on~$\Omega$.
\qed\end{proof}
\if{
\AK{28/11/23 and 16/12/23.
I am using `localized' properties of the preference relation, i.e., exactly what is needed for this proposition.
\\
Some discussion of set-valued optimization and preferences as well as comparison with the existing definitions should be added.
I think the properties above a weaker than the corresponding ones in \cite[Definition~5.55]{Mor06.2}, aren't they?
Still, strong arguments are needed to show that these properties are natural and useful.
}
\NDC{14/1/24.
Condition (i) can be rewritten as: `there exists a $y\in L^\circ(0)$ such that $y\in L^\bullet_{\eps}(0)$'.
The corresponding condition in \cite[Definition~5.55]{Mor06.2} is:
`$y\in$cl$L(y)$ for all $y$ near $0$'
(they required `for all $y$ near $\by$', but in the current notation, I think it should be $0$ in place of $\by$).
This condition is equivalent to: `$y\in L^\bullet_{\eps}(0)$ for all $\varepsilon>0$ and $y$ near 0'.
Hence, I think condition (i) in Proposition~\ref{P4.2} is easier to check than the one in \cite[Definition~5.55]{Mor06.2} in the sense that it requires `the existence of a $y$' instead of `for all
$y$'.
Still, I have not been able to prove that the condition in Proposition~\ref{P4.2} is indeed weaker than the other one.
}
\AK{27/01/24.
These days I have been trying to check and recheck, and make sense of condition (i) as well as compare it with the condition
$y\in$cl$L(y)$ in \cite[Definition~5.55]{Mor06.2}.
I have constructed an example showing that the former is NOT stronger that the latter.

More importantly, it looks like condition (i) is the right one for the purpose.
At the same time, I am getting suspicious about the theory developed in \cite{MorTreZhu03} and then copied in \cite{Mor06.2,Bao14.2} and possibly other publications.
To apply their ``extended extremal principle" in vector or set-valued optimization, they need to prove an analogue of our Proposition~\ref{P4.2}.
I have only found a short sketch in \cite[Example~3.5]{MorTreZhu03}.
It lacks details and does not look convincing.
In \cite[Proposition~2]{Bao14.2}, Bao refers to a non-existing Example~5.56 in \cite{Mor06.2}.
Correct me please if my suspicions are unjustified.}

\NDC{28/1/24.
Let me have a closer look at the proof of \cite[Example~3.5]{MorTreZhu03}.
I think condition (i) in the above example appeared for the first time in \cite[Definition~3.1]{Zhu00}.
I think it is a typo in \cite{Bao14.2}.
I think Bao refered to Example 3.65.
}
}\fi

Thanks to Proposition~\ref{P4.2}, if the level-set mapping $L$ satisfies conditions \eqref{O1} and \eqref{O5},
then extremal points of problem~\eqref{P} satisfy the necessary conditions in Theorem~\ref{T4.1} and its corollaries.
\red{In particular, the next statement holds true.}

\begin{corollary}
\red{Let $X$ and $Y$ be Banach spaces, $\Omega$ and $\gph F$ be closed, $\bx\in\Omega$, $\by\in F(\bx)$, and $\de>0$.
Suppose that condition $(QC)_C$ is satisfied as well as conditions \eqref{O1} and \eqref{O5} for some mapping $L:Y\toto Y$.
If $(\bx,\by)$ is extremal for
$F$ on $\Omega$, then
there is an $M>0$ such that,
for any $\varepsilon>0$, there exist $(x_1,y_1)\in\gph F\cap B_{\varepsilon}(\bx,\by)$,
$x_2\in\Omega\cap B_{\varepsilon}(\bx)$,
$y\in L^-(\by)\cap B_\de(\by)$,
$y_2\in\cl L(y)\cap B_\eps(\by)$,
and $y^*\in N^C_{\cl L(y)}(y_2)+\eps\B_{Y^*}$
such that $\|y^*\|=1$, and condition \eqref{C4.2-1} holds true.}
\sloppy

\red{If $X$ is Asplund and condition $(QC)_F$ is satisfied,
then the above assertion holds true with $N^F$ and $D^{*F}$ in place of $N^C$ and $D^{*C}$, respectively.}
\end{corollary}

\if{
\AK{5/03/24.
At the moment, we have no applications of the ``level-set mapping'' theory here and in the next section.
What is it for?
Do we improve anything in the literature?
}

Next, we briefly discuss a seemingly more general constrained set-valued minimization problem:
\begin{gather}
\label{P1}
\tag{$\mathcal{P}$}
\text{minimize }\;F(x)\quad \text{subject to }\; G(x)\cap K\ne\es,\; x\in\Omega,
\end{gather}
where $F:X\toto Y$ and $G:X\toto Z$ are mappings between normed spaces, $\Omega\subset X$, $K\subset Z$, and $Y$ is equipped with a level-set mapping $L$.

\if{
\NDC{10/1/24.
Could it be another letter, for instance $(Q)$, instead of $(P_1)$ since we do not have $(P_2)$?
}
\AK{11/01/24.
I am not sure.
$P$ is common notation because it comes from ``Problem".
How to explain $Q$?
This could be misleading for some readers.
However, my objections are not very strong.}
\NDC{11/1/24.
How about $(\hat P)$?
$(\hat P)$ is better than $(\widehat P)$.}
\Thao{12/1/2024. How about $(GP)$ indicating `General version of $(P)$'? I am ok with any of your suggestions here :-).}
}\fi

Depending on the choice of the space $Z$ and subset $K$, the constraint $G(x)\cap K\ne\es$ can model a system of equalities and inequalities as well as more general operator-type constraints.
Of course, problem \eqref{P1} can be treated as a particular case of problem \eqref{P} with the set of \emph{admissible solutions} $\widehat\Omega:=\{x\in\Omega\mid G(x)\cap K\ne\es\}$ in place of $\Omega$.

If $(\bx,\by)$ is locally extremal for $F$ on $\widehat\Omega$, then $\bx\in\Omega$, $\by\in F(\bx)$,
and there exists a $\bz\in G(\bx)\cap K$.
If in addition the level-set mapping $L$ satisfies conditions \eqref{O1} and \eqref{O5}, then, by Proposition~\ref{P4.2}, $\{F,\widehat\Omega,L(\by)\cup\{\by\}\}$ is locally extremal at $(\bx,\by)$ with respect to $\{\cl L(y)\mid y\in Y\}$.
By Proposition~\ref{P4.1}, taking into account definition \eqref{Lbul}, this means that {there is a $\rho>0$ such that}, for any $\varepsilon>0$, there exists a $y\in Y$ with $d(\by,L(y))<\eps$ satisfying $F(\widehat\Omega\cap B_\rho(\bx))\cap\cl L(y)\cap B_\rho(\by)=\emptyset$,
or equivalently,
\begin{gather*}
\widehat F(\Omega\cap B_\rho(\bx))\cap((\cl L(y)\cap B_\rho(\by))\times K)=\emptyset,
\end{gather*}
where $\widehat F:=(F,G):X\toto Y\times Z$.
Note that $(\by,\bz)\in\widehat K$, where $\widehat K:=(L(\by)\cup\{\by\})\times K$, and $d((\by,\bz),\cl L(y)\times K)=d(\by,L(y))<\eps$.
By Definition~\ref{D4.1}, this shows that the triple $\{\widehat F,\Omega,\widehat K\}$ is locally extremal at $(\bx,(\by,\bz))$ with respect to $\widehat\Xi:=\{\cl L(y)\times K\mid y\in Y\}$.
Hence, locally extremal points of problem~\eqref{P1} satisfy the necessary conditions in Theorem~\ref{T4.1} and its corollaries with $\widehat F$, $\widehat K$ and $\widehat\Xi$ in place of $F$, $K$ and $\Xi$, respectively.
For instance, application of Corollary~\ref{C4.2} yields the following theorem of alternative.

\begin{theorem}
\label{T4.2}
Let $X$, $Y$ and $Z$ be Banach spaces, $\Omega$, $K$, $\gph F$ and $\gph G$ be closed.
If $(\bx,\by)$ is locally extremal in problem \eqref{P1}, and $\bz\in G(\bx)\cap K$, then one of the following assertions holds true:
\begin{enumerate}
\item
\label{T4.2.1}
there is an $M>0$ such that,
for any $\varepsilon>0$, there exist $(x_1,y_1,z_1)\in\gph (F,G)\cap B_{\varepsilon}(\bx,\by,\bz)$, $x_2\in\Omega\cap B_{\varepsilon}(\bx)$,
$y\in Y$,
$y_2\in\cl{L(y)}\cap B_\eps(\by)$,
$z_2\in K {\cap B_{\varepsilon}(\bz)}$,
$(y^*,z^*)\in N^C_{\cl L(y)\times K}(y_2,z_2)+\eps\B_{Y^*\times Z^*}$
such that $\|(y^*,z^*)\|=1$
and
\begin{gather*}
0\in D^{*C}(F,G)(x_1,y_1,z_1)(y^*,z^*) +N^C_\Omega(x_2)\cap(M\B_{X^*})+\eps\B_{X^*};
\end{gather*}
\if{
\NDC{12/1/24.
Should it be written as
$$0\in D^{*C}(F,G)(x_1,y_1,z_1)(B_\varepsilon(y^*,z^*))+N^C_\Omega(x_2)+\eps\B_{X^*}?$$
}}\fi
\item
\label{T4.2.2}
for any $\varepsilon>0$, there exist $(x_1,y_1,z_1)\in\gph (F,G)\cap B_{\varepsilon}(\bx,\by,\bz)$,
$x_2\in\Omega\cap B_{\varepsilon}(\bx)$,
and $x^*\in D^{*C}(F,G)(x_1,y_1,z_1) (\eps\B_{{Y^* \times Z^*}})$
such that $\|x^*\|=1$ and $d\left(-x^*,N^C_\Omega(x_2)\right)<\eps$.
\end{enumerate}
If $X$ is Asplund, then $N^C$ and $D^{*C}$ in the above assertions can be replaced by $N^F$ and $D^{*F}$, respectively.
\end{theorem}
\NDC{7/1/24. In view of Remark 4.4, I think \cite[Corollary 3.1 and Theorem~3.2]{ZheNg06} are consequences of Theorem~\ref{T4.2}.}
\AK{8/01/24.
Should this remark be `upgraded' into a theorem?}
\if{
\AK{9/01/24.
Instead of adding another theorem,  Theorem~\ref{T4.1} can be upgraded.
I hesitate going this way as this would make the presentation more complicated.
Ideally, it would be good to formulate the ``additional calculus'' mentioned in the remark.
It should not be difficult to prove the needed result using the conventional extremal principle, but this would be a deviation from the mainstream of this paper.
It would be better to find such a result somewhere.
I cannot remember ever proving it myself or having seen it, but my memory is not reliable. :-(}
}\fi

\if{
\NDC{10/1/24.
I think the calculus rule for the coderivative of the combined mapping $(F,G)$ can be found in \cite[Theorem~3.18]{Mor06.2}.
The calculus rule was established under the a qualification condition and a PSNC assumption.
Applying \cite[Theorem~3.18]{Mor06.2} (by imposing additional assumption on $F$ and $G$) to the above theorem leads to an upgraded version of Theorem~\ref{T4.2}.
I think it is OK to mention all the observations as a remark.
At the same time, I am happy to formulate the theorem if you decide to go this way.
}
}\fi
\AK{11/01/24.
Theorem~3.18 is actually in the first volume.
We cannot refer to it as it is formulated for limiting coderivatives.
The \Fr\ coderivative version should be true without PSNC.
It could make sense checking the earlier papers by Mordukhovich for such a result.
See the references on page 366 of the book.}
\NDC{11/1/24.
Thank you for your instruction.
The Fr\'echet version of \cite[Theorem~3.18]{Mor06.2} can be found in \cite[Theorem~4.8]{Mor97}.
Are you aware of the existence of this kind of result for the Clarke coderivative?
}

\blue{
The next statement presents a calculus rule for the Fr\'echet coderivative of the combined mapping $(F,G)$; cf. \cite[Theorem~3.18]{Mor06.2}.}

\begin{lemma}\label{L4.1}
Let $X,Y$ and $Z$ be Asplund spaces, $\gph F$ and $\gph G$ be closed, $(\bx,\by,\bz)\in\gph (F,G)$, $y^*\in Y^*$, $z^*\in Z^*$.
Suppose $F$ and $G$ have the Aubin property at $(\bx,\by)$ and $(\bx,\bz)$, respectively.
Then, for any $\varepsilon>0$, there exist $(x_1,y)\in\gph F\cap B_\varepsilon(\bx,\by)$, $(x_2,z)\in\gph G\cap B_\varepsilon(\bx,\bz)$,
and $(\hat y^*,\hat z^*)\in B_\varepsilon(y^*,z^*)$ such that
\begin{gather*}
D^{*F}(F,G)(\bx,\by,\bz)(y^*,z^*)\subset D^{*F}F(x_1,y)(\hat y^*)+D^{*F}G(x_2,z)(\hat z^*)+\varepsilon\B_{X^*}.
\end{gather*}
\end{lemma}

The next statement is a consequence of Theorem~\ref{T4.2}.
\begin{corollary}\label{C4.4}
Let $X,Y$ and $Z$ be Asplund spaces, $(\bx,\by,\bz)\in\gph (F,G)$, $\Omega$, $K$, $\gph F$ and $\gph G$ be closed, $F$ and $G$ have the Aubin property around $(\bx,\by)$ and $(\bx,\bz)$, respectively.
If $(\bx,\by)$ is locally extremal in problem \eqref{P1}, then one of the following assertions holds true:
\begin{enumerate}
\item
there is an $M>0$ such that, for any $\varepsilon>0$, there exist
$(x_1,y)\in\gph F\cap B_{\varepsilon}(\bx,\by)$,
$(x_2,z)\in\gph G\cap B_{\varepsilon}(\bx,\bz)$, $\hat x\in\Omega\cap B_{\varepsilon}(\bx)$,
$y'\in Y$,
$\hat y \in\cl{L(y')}\cap(\eps\B_Y)$,
$\hat z\in K {\cap B_{\varepsilon}(\bz)}$,
$(y^*,z^*)\in N^F_{\cl{L(y')\times K}}(\hat y,\hat z)+\eps\B_{Y^*\times Z^*}$
with $\|(y^*,z^*)\|=1$, and $(\hat y^*,\hat z^*)\in B_{\varepsilon}(y^*, z^*)$ satisfying
\sloppy
\begin{gather}\label{C4.4-1}
0\in D^{*F}F(x_1,y)(\hat y^*)+D^{*F}G(x_2,z)(\hat z^*)+N^F_{\Omega}(\hat x)\cap(M\B_{X^*})+\varepsilon\B_{X^*}.
\end{gather}
\item
for any $\varepsilon>0$, there exist
$(x_1,y)\in\gph F\cap B_{\varepsilon}(\bx,\by)$,
$(x_2,z)\in\gph G\cap B_{\varepsilon}(\bx,\bz)$,
$\hat x\in\Omega\cap B_{\varepsilon}(\bx)$,
and $x^*\in D^{*F}F(x_1,y)(y^*)+D^{*F}G(x_2,z)(z^*)+\varepsilon\B_{X^*}$
such that $\|x^*\|=1$ and $d\left(-x^*,N^F_\Omega(x_2)\right)<\eps$.
\end{enumerate}
\end{corollary}

\NDC{12/1/24.
I think if we apply Lemma~\ref{L4.1} to Corollary~\ref{C4.4} with $G=(F_1,\ldots,F_n)$, then we will arrive at
\cite[Theorem~3.2]{ZheNg06}.
Should we go this way or add this observation (if it is true!) as a remark?
}

\begin{proof}
Suppose $(\bx,\by)$ is locally extremal in problem \eqref{P1}.
\begin{enumerate}
\item
By Theorem~\ref{T4.2}\,\ref{T4.2.1}, there is an $M>0$ such that, for any $\varepsilon>0$, one can find $(x'_1,y'_1,z'_1)\in\gph (F,G)\cap B_{\varepsilon/2}(\bx,\by,\bz)$, $\hat x\in\Omega\cap B_{\varepsilon/2}(\bx)$,
$y'\in Y$,
$\hat y\in\cl{L(y')}\cap(\frac{\varepsilon}{2}\B_Y)$,
$\hat z\in K {\cap B_{\varepsilon/2}(\bz)}$,
$(y^*,z^*)\in N^F_{\cl{L(y')\times K}}(\hat y,\hat z)+\eps\B_{Y^*\times Z^*}$
with $\|(y^*,z^*)\|=1$
such that
\sloppy
\begin{gather}\label{C4.4-3}
0\in D^{*F}(F,G)(x'_1,y'_1,z'_1)(y^*,z^*)+N^F_\Omega(\hat x)\cap(M\B_{X^*})+(\varepsilon/2)\B_{X^*}.
\end{gather}
In view of Lemma~\ref{L4.1},
there exist $(x_1,y)\in\gph F\cap B_{\varepsilon/2}(x'_1,y'_1)$,
$(x_2,z)\in\gph G\cap B_{\varepsilon/2}(x'_1,z'_1)$, and $(\hat y^*,\hat z^*)\in B_{\varepsilon/2}(y^*,z^*)$
satisfying
\begin{gather}\label{C4.4-4}
D^{*F}(F,G)(x'_1,y'_1,z'_1)(y^*,z^*)\subset D^{*F}F(x_1,y)(\hat y^*)+D^{*F}G(x_2,z)(\hat z^*)+(\varepsilon/2)\B_{X^*}.
\end{gather}
It is clear that $(x_1,y)\in\gph F\cap B_{\varepsilon}(\bx,\by)$, and
$(x_2,z)\in\gph G\cap B_{\varepsilon}(\bx,\bz)$.
Then \eqref{C4.4-1} follows from \eqref{C4.4-3}, \eqref{C4.4-4}, and the fact $(\varepsilon/2)\B_{X^*}+(\varepsilon/2)\B_{X^*}=\varepsilon \B_{X^*}$.
\item
By Theorem~\ref{T4.2}\,\ref{T4.2.2}, for any $\varepsilon>0$ there exist $(x',y',z')\in\gph (F,G)\cap B_{\varepsilon/2}(\bx,\by,\bz)$,
$\hat x\in\Omega\cap B_{\varepsilon/2}(\bx)$,
$(\hat y^*,\hat z^*)\in (\eps/2)\B_{{Y^* \times Z^*}}$,
$x^*\in D^{*C}(F,G)(x',y',z')(\hat y^*,\hat z^*)$
such that $\|x^*\|=1$ and $d\left(-x^*,N^C_\Omega(\hat x)\right)<\eps/2$.
By Lemma~\ref{L4.1},
there exist $(x_1,y)\in\gph F\cap B_{\varepsilon/2}(x',y')$,
$(x_2,z)\in\gph G\cap B_{\varepsilon/2}(x',z')$, $(y^*,z^*)\in B_{\varepsilon/2}(\hat y^*,\hat z^*)$
such that
\begin{gather}\label{C4.4-6}
D^{*F}(F,G)(x',y',z')(\hat y^*,\hat z^*)\subset D^{*F}F(x_1,y)(y^*)+D^{*F}G(x_2,z)(z^*)+(\varepsilon/2)\B_{X^*}.
\end{gather}
It is clear that $(x_1,y)\in\gph F\cap B_{\varepsilon}(\bx,\by)$,
$(x_2,z)\in\gph G\cap B_{\varepsilon}(\bx,\bz)$, and $(y^*,z^*)\in \varepsilon\B_{Y^*\times Z^*}$.
\end{enumerate}
\qed\end{proof}

\NDC{21/1/24.
I think Theorem~\ref{T4.2} and Corollary~\ref{C4.4} can be combined in one theorem.
Should we do that?
}
\NDC{21/1/24.
I think they should not be combined since we will lose the Clarke case.}

\begin{remark}
The optimality conditions in Theorem~\ref{T4.2} are formulated in terms of coderivatives of the combined mapping $(F,G):X\toto Y\times Z$.
To obtain conditions in terms of coderivatives of the individual mappings $F$ and $G$, some additional calculus is needed.
\end{remark}
\begin{remark}
It is possible to upgrade slightly the model used in Definition~\ref{D4.1} and the subsequent statements: instead of a single mapping $F:X\toto Y$, one can consider a finite collection of mappings $F_i:X\toto Y_i$ $(i=1,\ldots,n)$ together with
a collection of sets $K_i\subset Y_i$ $(i=1,\ldots,n)$ and points $\by_i\in F_i(\bx)$ $(i=1,\ldots,n)$.
Such an extended model can still be easily embedded into the setting studied in Sect.~\ref{S4}.
Instead of the pair of sets \eqref{Om} in $X\times Y$, one would need to consider an appropriate family of $n+1$ sets in $X\times Y_1\times\ldots\times Y_n$.
Apart from upgrading Theorem~\ref{T4.2}, this way one can treat generalizations of problems  \eqref{P} and \eqref{P1} containing multiple constraints and multiple objective functions.
\end{remark}
}\fi

\section{Set-Valued Optimization: Multiple Mappings}
\label{S6}

It is not difficult to upgrade
the model used in Definition~\ref{D4.1} and the subsequent statements to make it directly applicable to constraint optimization problems: instead of a single mapping $F:X\toto Y$ with $\by\in F(\bx)$ for some $\bx\in\Omega\subset X$ and a single family $\Xi$ of subsets of $Y$, one can consider finite collections of mappings $F_i:X\toto Y_i$ between normed spaces
together with points $\by_i\in F_i(\bx)$, and nonempty families $\Xi_i$ of subsets of $Y_i$ $(i=1,\ldots,n)$.

This more general setting can be viewed as a structured particular case of the set-valued optimization model
considered in Sect.~\ref{S5} if one sets
\begin{gather*}
Y:=Y_1\times\ldots\times Y_n,\quad
F:=(F_1,\ldots,F_n),\quad
\by:=(\by_1,\ldots,\by_n)\AND \Xi:=\Xi_1\times\ldots\times\Xi_n.
\end{gather*}
Thus, $\by\in F(\bx)$, and $A\in\Xi$ means that $A=A_1\times\ldots\times A_n$ and $A_i\in\Xi_i$
$(i=1,\ldots,n)$.
To shorten the notation, we keep talking in this section about extremality/stationarity of the triple $\{F,\Omega,\Xi\}$ at $(\bx,\by)$.

\begin{definition}
\label{D4.1+}
The triple $\{F,\Omega,\Xi\}$ is
extremal (resp., stationary, approximately stationary) at $(\bx,\by)$ if the collection of $n+1$ families of sets:
\begin{gather*}
\widehat\Xi_i:=\{\Omega_i\}\;\;(i=1,\ldots,n) \AND \widehat\Xi_{n+1}:=\{\Omega\times A\mid A\in\Xi\}.
\end{gather*}
is extremal (resp., stationary, approximately stationary) at $(\bx,\by)$, where
$\Omega_i:=\{(x,y_1,\ldots,y_n)\in X\times Y_1\times\ldots\times Y_n\mid y_i\in F_i(x)\}$ $(i=1,\ldots,n)$.
\sloppy
\end{definition}

With the notation introduced above,
Definitions~\ref{D1.5} and \ref{D4.1+} lead to characterizations of the extremality and stationarity of the triple $\{F,\Omega,\Xi\}$ given in parts \eqref{P3.3.1} and \eqref{P3.3.2} of Proposition~\ref{P3.3}.
The corresponding characterization of the approximate stationarity is a little different.
It is formulated in the next proposition.

\begin{proposition}
The triple $\{F,\Omega,\Xi\}$
is
approximately stationary at $(\bx,\by)$ if and only if, for any $\varepsilon>0$, there exist a $\rho\in(0,\varepsilon)$,
$A_i\in\Xi_i$ $(i=1,\ldots,n)$,
$x_i\in B_\varepsilon(\bx)$ $(i=1,\ldots,n+1)$, and
$y_i,v_i\in B_\varepsilon(\by_i)$ $(i=1,\ldots,n)$ such that $d((x_i,y_i),\gph F_i)<\varepsilon\rho$, $d(v_i,A_i)<\varepsilon\rho$ $(i=1,\ldots,n)$,
$d(x_{n+1},\Omega)<\varepsilon\rho$ and,
for each $x\in\Omega\cap B_\rho(x_{n+1})$,
there is an $i\in\{1,\ldots,n\}$ such that
\begin{gather*}
F_i(x_i+x-x_{n+1})\cap \big(y_i+(A_i-v_{i})\cap(\rho\B_Y)\big)=\es.
\end{gather*}
\end{proposition}

Application of Theorem~\ref{T1.4} in the current setting produces necessary conditions for approximate stationarity and, hence, also stationarity and extremality extending Theorem~\ref{T4.1} and its corollaries.
Condition $({QC})_C$ can be extended as follows:
\begin{enumerate}
\item [ ]
\begin{enumerate}
\item [$(\widehat{QC})_C$]
there is an $\varepsilon>0$ such that
$\big\|\sum_{i=1}^{n+1}x_i^*\big\|\ge\eps$
for all
$(x_i,y_i)\in\gph F_i\cap B_{\varepsilon}(\bx,\by_i)$,
$x_i^*\in D^{*C}F_i(x_i,y_i)(\eps\B_{Y_i^*})$
$(i=1,\ldots,n)$,
$x_{n+1}\in\Omega\cap B_{\varepsilon}(\bx)$ and
$x_{n+1}^*\in N^C_\Omega(x_{n+1})$ such that $\sum_{i=1}^{n+1}\|x_i^*\|=1$,
\end{enumerate}
\end{enumerate}
while the corresponding extension $(\widehat{QC})_F$ of condition $({QC})_F$ is obtained by replacing $N^C$ and $D^{*C}$ in $(\widehat{QC})_C$ by $N^F$ and $D^{*F}$, respectively.
An extension of Corollary~\ref{C4.40} takes the following form.

\begin{theorem}
\label{T5.10}
Let $X$, $Y_1,\ldots,Y_n$ be Banach spaces, $\Omega$ and, for each $i=1,\ldots,n$, the graph $\gph F_i$ and all
members of $\Xi_i$ be closed.
Suppose
$\{F,\Omega,\Xi\}$ is approximately stationary at $(\bx,\by)$.
If condition $(\widehat{QC})_C$ is satisfied, then
there is an $M>0$ such that,
for any $\varepsilon>0$, there exist $(x_i,y_i)\in\gph F_i\cap B_{\varepsilon}(\bx,\by_i)$, $ A_i\in\Xi_i$,
$v_{i}\in A_i\cap B_\varepsilon(\by_i)$,
$y^*_i\in N^C_{ A_i}(v_i)+\eps\B_{Y_i^*}$ $(i=1,\ldots,n)$, and $x_{n+1}\in\Omega\cap B_{\varepsilon}(\bx)$
such that $\sum_{i=1}^n\|y_i^*\|=1$ and
\begin{gather*}
0\in\sum_{i=1}^n D^{*C}F_i(x_i,y_i)(y_i^*) +N^C_\Omega(x_{n+1})\cap(M\B_{X^*})+\eps\B_{X^*}.
\end{gather*}

If $X$ is Asplund and condition $(\widehat{QC})_F$ is satisfied,
then the above assertion holds true with $N^F$ and $D^{*F}$ in place of $N^C$ and $D^{*C}$, respectively.
\end{theorem}

\begin{remark}
\begin{enumerate}
\item
Proposition~\ref{P4.11} (with $F=(F_1,\ldots,F_n)$ in part (i)) gives two typical sufficient conditions for the fulfillment of conditions $(\widehat{QC})_C$ and $(\widehat{QC})_F$.
\item
Theorem~\ref{T5.10} covers \cite[Theorems~3.1 and 3.2]{ZheNg06}.
In view of the previous item,
it also covers \cite[Corollary 3.2]{ZheNg06}.
\item
Theorem~\ref{T5.10} is a consequence of the dual necessary conditions for approximate stationarity of a collection of sets in Theorem~\ref{T1.4}.
The latter theorem can be extended to cover a more general quantitative notion of approximate $\al$-sta\-tionarity (with a fixed $\al>0$), leading to corresponding extensions of Theorem~\ref{T5.10} and its corollaries covering, in particular, dual conditions for $\eps$-Pareto optimality in \cite[Theorems~4.3 and 4.5]{ZheNg11}.
\end{enumerate}
\end{remark}

\if
\AK{2/03/24.
Can you see a simple way of deducing from Theorem~\ref{T5.10} the fuzzy calculus result we discussed some time ago: a representation of elements of the coderivative of $(F_1,\ldots,F_n)$ via elements of the coderivatives of $F_1,\ldots,F_n$?
}
If, for each $i=1,\ldots,n$, the space $Y_i$ is equipped with a preference relation defined by an abstract level-set mapping $L_i:Y_i\toto Y_i$, we can consider a level-set mapping $L:Y\toto Y$ defined, for each $y=(y_1,\ldots,y_n)$, by $L(y):=L_1(y_1)\times\ldots\times L_n(y_n)$.
Definitions \eqref{L}, \eqref{Lbul} and \eqref{KXi} are applicable.
In particular, with $\by:=(\by_1,\ldots,\by_n)$ and $\de>0$, definition \eqref{KXi} gives
$K:=K_1\times\ldots\times K_n$ and $\Xi^\de:=\Xi_1^\de\times\ldots\times\Xi_n^\de$ where, for each $i=1,\ldots,n$,
\begin{gather*}
K_i:=L_i(\by_i)\cup\{\by_i\}\AND \Xi_i^\de:=\{\cl L_i(y_i)\mid y_i\in K_i\cap B_\de(\by_i)\}.
\end{gather*}
\fi

Employing the multiple-mapping model studied in this section, one can consider a more general than \eqref{P} optimization problem with set-valued constraints:
\begin{gather}
\label{P2}
\tag{$\mathcal{P}$}
\text{minimize }\;F_0(x)\quad \text{subject to }\; F_i(x)\cap K_i\ne\es\; (i=1,\ldots,n),\;\; x\in\Omega,
\end{gather}
where $F_i:X\toto Y_i$ $(i=0,\ldots,n)$ are mappings between normed spaces, $\Omega\subset X$, $K_i\subset Y_i$ $(i=1,\ldots,n)$, and $Y_0$ is equipped with a level-set mapping $L$.
The ``functional'' constraints in \eqref{P2} can model a system of equalities and inequalities as well as more general operator-type constraints.

Using the set of \emph{admissible solutions}
\begin{center}
$\widehat\Omega:=\{x\in\Omega\mid F_i(x)\cap K_i\ne\es,\; i=1,\ldots,n\}$,
\end{center}
we say that $(\bx,\by_0)\in X\times Y_0$ is
an extremal point of problem \eqref{P2} if it is
extremal for
$F_0$ on $\widehat\Omega$.
This means, in particular, that $\bx\in\Omega$, $\by_0\in F_0(\bx)$, and there exist $\by_i\in F_i(\bx)\cap K_i$ $(i=1,\ldots,n)$.

We are going to employ the model studied in the first part of this section with $n+1$ objects in place of $n$.
There are $n+1$ mappings $F_0,\ldots,F_n$ and $n$ sets $K_1,\ldots,K_n$ in \eqref{P2}.
As in \eqref{KXi}, we define
$\Xi_0^\de:=\{\cl L(y)\mid y\in L_\de(\by_0)\}$ $(\de>0)$,
where $L_\de(\by_0)=(L(\by_0)\cap B_\de(\by))\cup\{\by\}$.
Now, set
\begin{gather*}
Y:=Y_0\times\ldots\times Y_n,\;\;
F:=(F_0,\ldots,F_n),\;\;
\by:=(\by_0,\ldots,\by_n)\AND \Xi^\de:=\Xi_0^\de\times K_1\times\ldots\times K_n.
\end{gather*}
Using the same arguments, one can prove the next extension of Proposition~\ref{P4.2}.

\begin{proposition}
\label{P5.1}
Let $\bx\in\Omega$, $\by_0\in F_0(\bx)$, $\by_i\in F_i(\bx)\cap K_i$ $(i=1,\ldots,n)$, $\de>0$, and $F$, $\by$ and $\Xi^\de$ be defined as above.
Suppose $L$ satisfies conditions \eqref{O1} and \eqref{O5}.
If $(\bx,\by_0)$ is an extremal point of problem \eqref{P2}, then $\{F,\Omega,\Xi^\de\}$ is extremal at $(\bx,\by)$.
\end{proposition}

Condition $(\widehat{QC})_C$ in the current setting is reformulated as follows:
\begin{enumerate}
\item [ ]
\begin{enumerate}
\item [$(\widehat{QC})_C^\prime$]
there is an $\varepsilon>0$ such that
$\big\|\sum_{i=0}^{n+1}x_i^*\big\|\ge\eps$
for all
$(x_i,y_i)\in\gph F_i\cap B_{\varepsilon}(\bx,\by_i)$,
$x_i^*\in D^{*C}F_i(x_i,y_i)(\eps\B_{Y_i^*})$
$(i=0,\ldots,n)$,
$x_{n+1}\in\Omega\cap B_{\varepsilon}(\bx)$ and
$x_{n+1}^*\in N^C_\Omega(x_{n+1})$ such that $\sum_{i=0}^{n+1}\|x_i^*\|=1$,
\end{enumerate}
\end{enumerate}
while the corresponding reformulation $(\widehat{QC})_F^\prime$ of condition $(\widehat{QC})_F$ is obtained by replacing $N^C$ and $D^{*C}$ in $(\widehat{QC})_C^\prime$ by $N^F$ and $D^{*F}$, respectively.
In view of Proposition~\ref{P5.1}, Theorem~\ref{T5.10} yields the following statement.

\begin{corollary}
Let $X$, $Y_0,\ldots,Y_n$ be Banach spaces, the sets $\Omega$, $\gph F_i$ $(i=0,\ldots,n)$ and $K_i$ $(i=1,\ldots,n)$ be closed, and $\de>0$.
Suppose $L$ satisfies conditions \eqref{O1} and \eqref{O5}.
If $(\bx,\by_0)$ is an extremal point of problem \eqref{P2} and condition $(\widehat{QC})_C^\prime$ is satisfied,
then
there is an ${M>0}$ such that,
for any $\varepsilon>0$, there exist $(x_i,y_i)\in\gph F_i\cap B_{\varepsilon}(\bx,\by_i)$ $(i=0,\ldots,n)$,
$x_{n+1}\in\Omega\cap B_{\varepsilon}(\bx)$,
$y\in B_\de(\by_0)$,
$v_0\in\cl L(y)\cap B_\eps(\by_0)$,
$y_0^*\in N^C_{\cl L(y)}(v_0)+\eps\B_{Y_0^*}$,
$v_i\in K_i\cap B_\eps(\by_i)$
and $y_i^*\in N^C_{K_i}(v_i)+\eps\B_{Y_i^*}$ $(i=1,\ldots,n)$
such that $\sum_{i=0}^n\|y_i^*\|=1$ and
\begin{gather*}
0\in\sum_{i=0}^n D^{*C}F(x_i,y_i)(y_i^*) +N^C_\Omega(x_{n+1})\cap(M\B_{X^*})+\eps\B_{X^*}.
\end{gather*}

If $X$ is Asplund and condition $(\widehat{QC})_F^\prime$ is satisfied,
then the above assertion holds true with $N^F$ and $D^{*F}$ in place of $N^C$ and $D^{*C}$, respectively.
\end{corollary}

\noindent{\bf Acknowledgments }
A part of the work was done during Alexander Kruger's stay at the Vietnam Institute for Advanced Study in Mathematics in Hanoi.
He is grateful to the Institute for its hospitality and supportive environment.

\red{We thank the reviewers for their careful reading of the paper, detailed analysis of the approach and very specific comments and suggestions which led to significant changes in the paper improving its readability. We would also like to thank the handling editor for the perfect choice of the reviewers who are real experts in the field.}
\bigskip

\noindent{\bf Funding Information }
Nguyen Duy Cuong is supported by Vietnam National Program for the Development of Mathematics 2021-2030 under grant number B2023-CTT-09.
\bigskip

\noindent{\bf Data availability }
Data sharing is not applicable to this article as no datasets have been generated or analysed during the current study.
\bigskip

\section*{Declarations}

\noindent{\bf Competing Interests } The authors have no competing interests to declare that are relevant to the content of this article.

\bibliography{BUCH-kr,Kruger,KR-tmp}
\bibliographystyle{spmpsci}
\end{document}